\documentclass[a4paper,oneside,12pt]{article}

\usepackage{amsmath,amsthm,amsfonts,amscd,amssymb}
\usepackage{longtable,geometry}
\usepackage[english]{babel}
\usepackage[utf8]{inputenc}
\usepackage[active]{srcltx}
\usepackage[T1]{fontenc}
\usepackage{graphicx}
\usepackage{pstricks}
\usepackage{bbm}
\usepackage{mathtools}
\usepackage{MnSymbol}
\usepackage{stmaryrd}
\usepackage{nicefrac}
\usepackage{calrsfs}
\usepackage{rotating}
\usepackage{xcolor}
\usepackage{framed}
\usepackage{hyperref}
\hypersetup{
    colorlinks=false,
    pdfborder={0 0 0},
}
\usepackage{enumitem}
\usepackage{footnote}

\colorlet{shadecolor}{blue!15}

\newtheorem{thm}{Theorem}[section]

\newtheorem{cor}[thm]{Corollary}
\newtheorem{lem}[thm]{Lemma}
\newtheorem{prop}[thm]{Proposition}
\newtheorem{definition}[thm]{Definition}

\newtheorem{rem}[thm]{Remark}

\numberwithin{equation}{section}

\renewenvironment{proof}[1][\relax]
  {\paragraph{Proof\ifx#1\relax\else~of #1\fi}}%
  {~\hfill$\square$\par\bigskip}
  
\newcommand{\calC}{\mathcal{C}}
\newcommand{\calD}{\mathcal{D}}
\newcommand{\calE}{\mathcal{E}}
\newcommand{\calF}{\mathcal{F}}

\newcommand{\calH}{\mathcal{H}}

\newcommand{\calK}{\mathcal{K}}

\newcommand{\calS}{\mathcal{S}}

\newcommand{\bbC}{\mathbb{C}}

\newcommand{\bbE}{\mathbb{E}}

\newcommand{\bbH}{\mathbb{H}}

\newcommand{\bbL}{\mathbb{L}}

\newcommand{\bbN}{\mathbb{N}}

\newcommand{\bbP}{\mathbb{P}}

\newcommand{\bbR}{\mathbb{R}}

\newcommand{\bbZ}{\mathbb{Z}}

\newcommand{\eps}{\varepsilon}

\newcommand{\rk}[1]{\bgroup\color{red}%
  \par\medskip\hrule\smallskip%
  \noindent\textbf{#1}%
  \par\smallskip\hrule\medskip\egroup}

\newcommand{\lra}{\leftrightarrow}
\newcommand{\xlra}{\xleftrightarrow}
\newcommand\concel[2]{\ooalign{$\hfil#1\mkern0mu/\hfil$\crcr$#1#2$}}  
\newcommand\nxlra[1]{\mathrel{\mathpalette\concel{\xleftrightarrow{#1}}}}

\renewcommand{\int}{\mathrm{in}}

\newcommand{\1}{\mathbbm{1}}

\newcommand{\La}{\Lambda}
\newcommand{\la}{\lambda}

\renewcommand{\colon}{\,:\,}

\usepackage{enumitem}
\setlist[itemize]{itemsep=1pt, topsep=4pt}
\setlist[enumerate]{itemsep=1pt, topsep=4pt}

\title{Structure of Gibbs measures for planar FK-percolation and Potts models}
\date{\today}
\author{
	Alexander Glazman
    	\thanks{Faculty of Mathematics,
    	University of Vienna,
	    Oskar-Morgenstern-Platz 1,
	    Vienna, 
	    Austria.
	    \url{alexander glazman@univie.ac.at}}
    ~and Ioan Manolescu
        \thanks{D\'epartement de Math\'ematiques, 
        Universit\'e de Fribourg, 
        23 Chemin du Mus\'ee, 
        CH-1700 Fribourg, 
        Switzerland.
	    \url{ioan.manolescu@unifr.ch}}
	}

\begin{document}
\maketitle

\begin{abstract}
	We prove that all Gibbs measures of the $q$-state Potts model on $\bbZ^2$ are linear combinations of the extremal measures obtained as thermodynamic limits under free or monochromatic boundary conditions.
	In particular all Gibbs measures are invariant under translations.
	This statement is new at points of first-order phase transition, that is at $T=T_{c}(q)$ when $q>4$.
	In this case the structure of Gibbs measures is the most complex in the sense that there exist $q+1$ distinct extremal measures.
	
	Most of the work is devoted to the FK-percolation model on $\bbZ^{2}$ with $q\geq 1$, where we prove that every Gibbs measure is a linear combination of the free and wired ones.
	The arguments are non-quantitative and follow the spirit of the seminal works of
	Aizenman~\cite{Aiz80} and Higuchi~\cite{Hig81}, which established the Gibbs structure for the two-dimensional Ising model.
	Infinite-range dependencies in FK-percolation (i.e., a weaker spatial Markov property) pose serious additional difficulties compared to the case of the Ising model. 
	For example, it is not automatic, albeit true, that thermodynamic limits are Gibbs. 
	The result for the Potts model is then derived using the Edwards--Sokal coupling and auto-duality. 
	The latter ingredient is necessary since applying the Edwards--Sokal procedure to a Gibbs measure for the Potts model does not automatically produce a Gibbs measure for FK-percolation.
	
	Finally, the proof is generic enough to adapt to the FK-percolation and Potts models on the triangular and hexagonal lattices and to the loop $O(n)$ model in the range of parameters for which its spin representation is positively associated.	
\end{abstract}

\tableofcontents

\section{Introduction}

The Potts model was introduced in 1952~\cite{Pot52} and has since become one of the most studied models of statistical mechanics. 
Each vertex of a graph is assigned one of $q$ spins (colours), with spins of adjacent vertices interacting with a strength depending on the temperature $T>0$ of the system. The classical Ising model is obtained when $q=2$. As in the Ising model, the Potts model becomes increasingly ordered as the temperature decreases, and a phase transition occurs on lattices $\bbZ^{d}$ with $d \geq 2$. The type of the transition depends on $q$ and $d$, which renders the Potts model an archetype of statistical mechanics model with both continuous and discontinuous phase transitions.

The phase transition and its type can be described by the structure of the set of {\em Gibbs measures}.
These are infinite-volume measures which satisfy the equation of Dobrushin~\cite{Dob72} and Lanford--Ruelle~\cite{LanRue69} (DLR): 
for every finite domain, conditionally on the configuration on its exterior, the measure inside the domain is given by the Potts specification with boundary conditions provided by the outside configuration.
The set of Gibbs measures is convex and is thus determined by its extremal points.

This work is restricted to the planar case, which we discuss henceforth.
In terms of Gibbs measures, the phase diagram when $q\geq 2$ may be described as follows:
for $T>T_{c}(q)$ the Gibbs measure is unique and exhibits exponential decay of correlations ({\em free measure}); 
when $T<T_{c}(q)$ there are exist $q$ distinct ordered Gibbs measures obtained as thermodynamic limits under constant-spin ({\em monochromatic}) boundary conditions~\cite{CoqDumIofVel14}.
The value $T_c(q)$ for which the phase transition occurs was determined in~\cite{BefDum12}.
At $T=T_{c}(q)$, when  $2 \leq q \leq 4$ the Gibbs measure is unique~\cite{DumSidTas17} ({\em continuous phase transition}); 
when $q>4$ the phase transition is discontinuous ~\cite{DumGagHar16,RaySpi20}, and there exist at least $q+1$ distinct Gibbs measures: 
one free measure obtained as a limit as $T\searrow T_{c}(q)$ and $q$ monochromatic measures obtained as limits as $T\nearrow T_{c}(q)$.
The goal of the present paper is to prove that in this last case, the above-mentioned Gibbs measures are the only extremal ones. 

For the Ising model, the set of Gibbs measures was determined in the celebrated works of Aizenman~\cite{Aiz80} and Higuchi~\cite{Hig81}: 
every such measure was shown to be a linear combination of the plus and minus states; the two being equal for $T \geq T_c(2)$. 
The original proofs relied on the Griffiths' inequality and other properties specific to the Ising model.
The later proof of Georgii--Higuchi~\cite{GeoHig00} is more robust, but still does not apply to the Potts model,
as it relies on positive association and symmetry of phases. 

In the present work, these obstacles are circumvented by considering the graphical representation of the Potts model 
known as Fortuin--Kasteleyn (FK) percolation or random-cluster model. 
In this model, configurations are assignments of open and closed states to each edge, 
and their probability depends on the numbers of open edges and of connected components.
The model has two parameters: an edge-intensity $p \in [0,1]$ akin to the temperature of the Potts model and a cluster-weight $q \geq 1$.
When $q \geq 2$ is an integer, the Edwards--Sokal coupling~\cite{EdwSok88} relates FK-percolation to the $q$-state Potts model. 

While FK-percolation satisfies positive association and exhibits a natural symmetry between the primal and dual models, 
the proofs previously developed for the Ising model do not apply as the interactions are not nearest-neighbour.
We prove nevertheless that all Gibbs measures for FK-percolation are linear combinations of the wired and the free measures.
Our arguments rely on the self-duality of the square grid, but may be extended to the triangular and hexagonal lattices via the star--triangle transformation.

The typical example of an extremal Gibbs measure which is neither free nor monochromatic is one where several regions with monochromatic behaviour coexist.
The interfaces between these regions need then to have bounded variations; such interfaces are called {\em localised}.
It was conjectured by Dobrushin \cite{Dob68} that interfaces in two-dimensional models are never localised. 
When the surface tension is positive, the Ornstein--Zernike theory~\cite{CamIofVel03, Iof98, IofVel10} 
may be used to prove Gaussian behaviour of interfaces. This strategy was successfully implemented for the Potts model with $T<T_{c}(q)$~\cite{CamIofVel08,CoqDumIofVel14, CoqVel12}, and showed that the only extremal Gibbs measures are the monochromatic ones. 
Unfortunately this approach is still beyond reach for the Potts and FK-percolation models at the point of a discontinuous phase transition. 
While it is known in these situations that the correlation length is finite, its relation to the surface tension is not sufficiently well understood to allow the development of the Ornstein--Zernike theory.
It should nevertheless be mentioned that for $q$ very large, the Pirogov-Sinai theory applies to the critical FK-percolation and Potts models~\cite{LaaMesMir91} and may be used~\cite{DobShl85} to characterise the structure of Gibbs measures.

It is worth mentioning that in higher dimensions, interfaces may be localised: at low temperatures the Ising model was shown to exhibit a so-called Dobrushin state~\cite{Dob72} (see also~\cite{Bei75} for a simpler proof).

In closing, let us mention that the approach developed here for FK-percolation also applies to the loop $O(n)$ model on the hexagonal lattice~\cite{DomMukNie81}, 
in the regime of positive association. 
The configurations here are collections of non-intersecting simple cycles (loops) and their probabilities depend on two parameters: the loop-weight $n$ and the edge-weight $x$. When $n\geq 1$ and $nx^{2}\leq 1$, the spin representation of the model was shown to be positively associated~\cite{DumGlaPel21}, 
and the system exhibits either exponential tails for the lengths of loops ({\em small loops}), or scale-invariant behaviour ({\em macroscopic loops}). 
These regimes are reminiscent of the discontinuous and continuous phase transitions of FK-percolation, respectively.
In the former case, our approach shows that every Gibbs measure for the spin representation is a linear combination of the plus and the minus states;
in the latter case, there exists a unique Gibbs measure for the spin representation.
Since the loops are invariant under global sign flip of the spins, the Gibbs measure for the loop $O(n)$ model is unique for all $n\geq 1$ and $x\leq 1/\sqrt n$.
	
\subsection{Results for the Potts model}

We work on the square lattice, with vertex set $\bbZ^2$ and edge-set $E(\bbZ^{2})$ connecting vertices at euclidian distance $1$. 
For a finite subgraph $D = (V,E)$ of $\bbZ^2$ (also referred to as {\em domain} henceforth), 
write $\partial D = \{v \in V \,: \, d_D(v) < 4\}$ and $D^{c} = \{v\in \bbZ^{2}\setminus V\}$.
Let $q\geq 1$ be integer, $T>0$ and $\tau \in \{0,1,\dots,q\}^{\partial D}$.
The $q$-state Potts model on $D$ at temperature $T$ with boundary conditions $\tau$ is the probability measure $\mu_{D,T,q}^\tau$
supported on the set of $\sigma\in \{1,2,\dots,q\}^{V}$ such that $\sigma_{u} = \tau_{u}$ for all $u\in \partial D$ with $\tau_{u}\neq 0$, 
and defined by
\begin{equation}\label{eq:def-potts}
	\mu_{D,T,q}^\tau[\sigma]=\frac{1}{Z^\tau(D,T,q)} \cdot \exp \big(-\tfrac{1}{T}\cdot \sum_{u\sim v} \1_{\{\sigma_{u} \neq \sigma_{v}\}}\big),
\end{equation}
where the sum is taken over pairs of adjacent vertices $u,v \in V$ and  
the normalising constant $Z^\tau(D,T,q)$ is called the {\em partition function}.

If $\tau$ takes a constant non-zero value on $\partial D$, the boundary conditions are called {\em monochromatic}.
For $i=1,\dots,q$, we use the notation $\mu_{D,T,q}^i$ for the measure $\mu_{D,T,q}^\tau $ with $\tau \equiv i$.
If $\tau$ is constant zero on $\partial D$, the boundary conditions are called {\em free}.
We denote the corresponding measure by $\mu_{D,T,q}^{\rm free}$.

Next we consider infinite-volume measures, that is probability measures on $\{1,\dots,q\}^{\bbZ^{2}}$.

\begin{definition}
    A probability  measure $\mu$ on $\{1,\dots,q\}^{\bbZ^{2}}$ is called {\em Gibbs} for the Potts model with parameters $T >0$ and $q \geq2$ if,  
    for any finite $D\subset \bbZ^{2}$ and for every $\tau \in \{1,\dots,q\}^{\partial D}$
    for which the conditioning is not degenerate
      \begin{equation}\label{eq:DLR-Potts}\tag{DLR-Potts}
    	\mu[. \, | \, \sigma_{|\partial D} = \tau \text{ and } \calF(D^{c})] =\mu_{D,T,q}^\tau,
    \end{equation}
    where $\calF(D^{c})$ is the sigma-algebra of events depending only on the spins in $D^{c}$
    and the above refers to the restriction of the left-hand side to $D$. 
\end{definition}

A probability measure $\mu$ on $\{1,\dots,q\}^{\bbZ^{2}}$ is called {\em tail trivial} if for any event $A$ that does not depend on any finite number of spins, either $\mu(A) = 0$ or $\mu(A) = 1$.
The set of Gibbs measures for the Potts model with fixed $T$ and $q$ form a convex subset of probability measures on $\{1,\dots,q\}^{\bbZ^{2}}$. 
The extremal points of this convex set are exactly the tail trivial Gibbs measures~\cite[Thm. 7.7]{Geo11}.

The monochromatic and free infinite-volume measures are defined as thermodynamic limits of the finite-volume measures under the corresponding boundary conditions.
Indeed, let $D_{n}$ be an increasing sequence of finite subgraphs of $\bbZ^{2}$ that eventually cover all vertices and edges of $\bbZ^{2}$; we say that $(D_{n})_{n\geq 1}$ increases to $\bbZ^2$.
It is standard (see~\cite[Prop 6.9]{GeoHagMae01}) that, as $n$ tends to infinity, the measures defined by monochromatic or by free boundary conditions have weak limits that do not depend on $D_{n}$ . We will use the following notation for them:
\[
	\mu_{T,q}^i := \lim_{n\to \infty}\mu_{D_{n},T,q}^i \quad \text{and} \quad \mu_{T,q}^{\rm free} := \lim_{n\to \infty}\mu_{D_{n},T,q}^{\rm free}
\]
As thermodynamic limits, $\mu_{T,q}^i$ and $\mu_{T,q}^{\rm free}$ satisfy~\eqref{eq:DLR-Potts} and hence are Gibbs measures.
More generally, it is an immediate consequence of the Spatial Markov property~\eqref{eq:SMP}  of the Potts model
that the weak limit of any converging sequence of measures $\mu_{D_{n},T,q}^{\tau_n}$ is Gibbs. 

It is standard that the monochromatic measures $\mu_{T,q}^i$ ($i = 1,\dots,q$) are always extremal (see Section~\ref{sec:ES}), though they may all be equal. 
The free measure is not always extremal, e.g. in the case $T < T_c$ it is the uniform convex combination of the distinct monochromatic ones (see Remark~\ref{rem:potts-known} below). 

The next theorem is the main result of this work and it shows that when $q\geq 2$, every Gibbs measure for the Potts model is a linear combination of the (possibly coinciding) $q+1$ measures defined above: $q$ wired and one free.

\begin{thm}\label{thm:mainPotts}
	Let $\mu$ be a Gibbs measure for the Potts model on $\bbZ^2$ with $q\geq 2$ colours at temperature $T>0$. 
	Then there exist $\la_0,\dots, \la_q  \geq 0$ summing up to $1$ such that
	\begin{align*}
	\mu = \la_0 \mu_{T,q}^{\rm free} +  \la_1 \mu_{T,q}^1 +\dots +  \la_q \mu_{T,q}^q.
	\end{align*}
\end{thm}

\begin{rem}\label{rem:potts-known}
	The result above is only new for the case $T = T_c(q)$ and $q$ such that the phase transition is discontinuous. 
	Indeed, for  $T > T_{c}(q)$ as well as for $T = T_{c}(q)$ for a continuous phase transition, the Gibbs measure is unique. 
	When $T < T_{c}(q)$, the extremal Gibbs measures were identified in~\cite{CoqDumIofVel14}:
	these are the $q$ monochromatic measures $(\mu_{T}^i)_{i =  1,\dots, q}$. 
	In particular $\mu_{T,q}^{\rm free} = \sum_{i = 1}^q \tfrac1q \mu_{T,q}^i$.
	
	As discussed below, at the point of a discontinuous phase transition, there exist exactly $q + 1$ extremal measures, 
	namely  $\mu_{T,q}^{\rm free}$ and the $q$ monochromatic measures $\mu_{T,q}^1,\dots, \mu_{T,q}^q$. 
\end{rem}

As mentioned in the introduction, for $q = 2$ (which corresponds to the Ising model), 
Theorem~\ref{thm:mainPotts} was proven in~\cite{Aiz80,Hig81} and again in ~\cite{GeoHig00}; previously \cite{GalMir72,MesMir75} treated the case of translation invariant measures for the Ising model.
Unfortunately, these arguments do not generalise to higher values of $q$, which is why we will prove Theorem~\ref{thm:mainPotts} using the FK representation. 
We should however mention that the works mentioned above served as inspiration for our proof of Theorem~\ref{thm:mainFK} below.

Finally, Theorem~\ref{thm:mainPotts} adapts to the Potts model on the triangular and hexagonal lattices with only minor modifications. See Section~\ref{sec:Hex} for details.

\subsection{Results for FK-percolation}
\label{sec:intro-fk}

Let $D = (V,E)$ be a subgraph of $\bbZ^2$.
A percolation configuration $\omega$ on $D$ is an element of $\{0,1\}^{E}$. 
An edge $e$ is said to be {\em open} (in $\omega$) if $\omega_e=1$, otherwise it is {\em closed}. 
When speaking of connections and connected components in $\omega$, we view it as a spanning subgraph of $D$ with edge-set $\{e\in E:\omega_e=1\}$.
For sets of vertices $A$ and $B$, we say that $A$ is connected to $B$ if there exists a path of edges of $\omega$ that connects a vertex of $A$ to a vertex of $B$. This event is denoted by $A \xlra{} B$ or $A \xlra{\omega} B$.
We also speak of connections in a set of edges~$C$ if all the edges of the path are contained in~$C$; we denote this by $\xlra{C}$. 

A boundary condition $\xi$ on $ D$ is a partition of $\partial D$.
Vertices of $\partial D$ that are in the same component of $\xi$ are said to be wired together. 
When all vertices are wired together, the conditions are called {\em wired} and are denoted by $1$. 
When no two vertices are wired together, the conditions are called {\em free} and are denoted by $0$.

When $D$ is finite, the FK-percolation measure on $D$ with edge-intensity $p \in (0,1)$, 
cluster-weight $q\geq 1$ and boundary conditions $\xi$ is the probability measure 
supported on percolation configurations on $D$ given by
\begin{equation}\label{eq:RCM_def1}
	\phi_{D,p,q}^\xi[\omega]=\frac{1}{Z^\xi(D,p,q)} (\tfrac{p}{1-p})^{|\omega|}q^{k(\omega^{\xi})},
\end{equation}
where $|\omega|= \# \{e\in E \colon \omega_{e} = 1\}$ is the number of open edges in $\omega$,
$\omega^{\xi}$ is the graph obtained from $\omega$ by identifying boundary vertices that are wired together in $\xi$, 
and $k(\omega^{\xi})$ is the number of connected components in $\omega^{\xi}$; 
the normalising constant $Z^\xi(D,p,q)$ is called the {\em partition function}.

We now consider percolation configurations defined on the entire square grid, that is $\omega  \in \{0,1\}^{E(\bbZ^{2})}$.
For a finite domain $D$, such $\omega$ induces boundary conditions $\xi$ on $\partial D$ defined as follows.
We say that two vertices $u,v\in\partial D$ are wired together in $\xi$ if and only if at least one of the following occurs: either $u\xlra{\omega\setminus E} v$, or $u\xlra{\omega\setminus E} \infty$ and $v\xlra{\omega\setminus E} \infty$.
We use the notation $u \xlra{\omega\setminus E} \infty$ for the event that $u$ belongs to an infinite cluster in $\omega\setminus E$.

\begin{definition}
    A measure $\phi$ on $\{0,1\}^{E(\bbZ^{2})}$ is called {\em Gibbs} for FK-percolation with parameters $p \in (0,1)$ and $q \geq 1$ if,  
    for any finite subgraph $D = (V,E)$ of $\bbZ^{2}$ and for every boundary condition $\xi$ on $D$
    for which the conditioning is not degenerate
    \begin{equation}\label{eq:DLR-FK}\tag{DLR-FK}
    	\phi[. \, | \, \omega \text{ induces $\xi$ and } \calF(E^{c})] =\phi_{D,p,q}^\xi,
    \end{equation}
    where $\calF(E^{c})$ is the sigma-algebra of events depending only on the edges in $E(\bbZ^2) \setminus E$
    and the above refers to the restriction of the left-hand side to $E$. 
\end{definition}

\begin{rem}\label{rem:DLR}
	In our definition of FK-percolation, if $\omega$ exhibits several infinite clusters, these are considered as a single one -- that is, they are wired at infinity.
	This is an arbitrary choice, but a posteriori bares no importance. 
	Indeed,  with our convention, we will prove that for any Gibbs measure there exists a.s. at most one infinite cluster.
	
	A more general definition would be to weight infinite clusters by some $q' \in [1,q]$. 
	With any such convention, Lemma~\ref{lem:Kcross} shows that the number of infinite clusters is uniformly bounded.
	Thus such a change of convention affects the measure by a bounded Radon-Nikodym derivative. 
	Hence, the uniqueness of the infinite cluster remains a.s. true, and the change of convention has no importance. 

	We mention nevertheless that with our convention, the dual of a Gibbs measure is not a priori Gibbs. 
\end{rem}

Let $(D_{n})_{n\geq 1}$ be a sequence of domains increasing to $\bbZ^{2}$.
The monotonicity properties of FK-percolation imply in a standard way (see Section~\ref{sec:fk-background}) 
that the sequences of measures $\phi_{D_n,p,q}^0$ and $\phi_{D_n,p,q}^1$ converge weakly;
their respective limits $\phi^0_{p,q}$ and  $\phi^1_{p,q}$ are called the free and wired infinite-volume measures. 
They are the minimal and maximal, respectively, Gibbs measures with parameters $(p,q)$ for the stochastic ordering, and in particular they are tail trivial.
Indeed, for any FK-percolation Gibbs measure $\phi$ with parameters $p\in (0,1)$ and $q \geq 1$,
\begin{align}\label{eq:0phi1}
	\phi_{p,q}^0 \leq_{\rm st}\phi \leq_{\rm st}\phi_{p,q}^1.
\end{align}
The next theorem states that $\phi^0_{p,q}$ and  $\phi^1_{p,q}$ are the only extremal Gibbs measures for FK-percolation.  

\begin{thm}\label{thm:mainFK}
	Let $\phi$ be a Gibbs measure for FK-percolation on $\bbZ^2$ with edge-intensity $p \in (0,1)$ and cluster-weight $q \geq 1$. 
	Then there exists $\la \in [0,1]$ such that 
	\begin{align*}
		\phi = \la \phi_{p,q}^0 + (1-\lambda)  \phi_{p,q}^1.
	\end{align*}
\end{thm}

\begin{rem}
	Except when $p = p_c(q) = \frac{\sqrt q}{1 + \sqrt q}$ and $q > 4$, it is known that $\phi_{p,q}^0 =  \phi_{p,q}^1$,
	and it follows directly from \eqref{eq:0phi1} that the infinite-volume measure is unique  
\end{rem}

This result should be compared to its counterpart for the Ising model~\cite{Aiz80, Hig81,GeoHig00}. 
The main reason why these proofs do not extend directly is the appearance of the number of clusters in the definition of the measure, which induces infinite-range dependencies. More precisely, the spatial Markov property in FK-percolation is significantly weaker than that of the Ising model: the distribution on a finite subgraph $D$ is not determined by the set of open/closed edges on the boundary of $D$ --- instead the boundary condition is determined by the configuration on the entire complement of $D$. Nevertheless, similarly to the Ising model, FK-percolation on $\bbZ^{2}$ is positively associated and self-dual at $p=p_{c}(q)$~\cite{BefDum12}. The latter property replaces the symmetry between $+$ and $-$ spins that was used in the previous works. 

Our approach follows the ideas of~\cite{GeoHig00}; we sketch it below so as to stress the difference with the Ising model.
First note that the result may be obtained quite easily for translation invariant measures (see Section~\ref{sec:translation_inv}). 

Let us proceed by contradiction and assume the existence of an extremal Gibbs measure $\phi$, which is different from $\phi^0$ and $\phi^1$. 
Then $\phi$ is tail trivial and must exhibit at least one infinite primal cluster and one infinite dual cluster. The proof of~\cite{GeoHig00} extends to show that in each half-plane there exists at most one infinite primal cluster and one infinite dual one. This implies the existence of one or two bi-infinite interfaces separating primal and dual clusters in the plane. These interfaces must have bounded fluctuations and thus visit a vicinity of $0$. The goal is then to show that the measure is invariant under translations, thus achieving a contradiction. 

Denote by $\phi'$ a shift of $\phi$ by an arbitrary integer vector. 
As in~\cite{GeoHig00}, we consider a duplicate system, which consists of two independent samples: $\omega$ from $\phi$ and $\omega'$ from $\phi'$. The core of the argument is the construction of arbitrarily large domains $D$ which are measurable in terms of the configurations outside of them and such that the boundary conditions on $D$ induced by $\omega'$ dominate those induced by $\omega$.
It then follows that $\phi'$ stochastically dominates $\phi$. The opposite bound is obtained by symmetry, and thus $\phi = \phi'$, 
which is to say that $\phi$ is invariant under translations, as required.

In~\cite[Lemma 5.5]{GeoHig00}, the construction of domains $D$ as above is quite direct, 
since in the Ising model it suffices to ensure that $\omega' \geq \omega$ on $\partial D$.
For FK-percolation however, that  $\omega' \geq \omega$ on the boundary of $D$ does not imply that the boundary conditions on $D$ induced by $\omega'$ dominates those induced by $\omega$.
Our construction is significantly more involved and requires a sophisticated exploration argument that allows to compare boundary conditions. The exploration relies on a new result for  percolation in the half-plane (Proposition~\ref{prop:hp_perco}), which may be of independent interest.

Due to the presence of long-range interactions in FK-percolation, it is not {\em a priori} clear that every thermodynamic limit is a Gibbs measure. 
However, this statement can be derived from Theorem~\ref{thm:mainFK} and Lemma~\ref{lem:Kcross} below.

\begin{cor}\label{cor:every-limit-is-gibbs}
	Let $p\in (0,1)$, $q\geq 1$, $(D_{n})_{n \geq1}$ be a sequence of domains increasing to $\bbZ^{2}$ 
	and $(\xi_{n})_{n \geq1}$ be a sequence of boundary conditions on $(D_{n})_{n \geq1}$.
	Assume that the sequence of FK-percolation measures $\phi_{D_{n},p,q}^{\xi_{n}}$ has a weak limit.
	Then the limiting measure is a Gibbs measure for FK-percolation with parameters $(p,q)$.
	Moreover, it is a linear combination of the free and the wired measures.
\end{cor}

Finally, let us mention that the results of this section adapt to FK-percolation on the triangular and hexagonal lattices as described in Section~\ref{sec:Hex}.

\subsection{Loop $O(n)$ model}

Let ${\rm Hex}$ denote the hexagonal lattice of Figure~\ref{fig:tri}. 
A loop configuration on ${\rm Hex}$ is any even subgraph of ${\rm Hex}$.
Then, each loop configuration consists of disjoint cycles (loops) and bi-infinite paths.
Let $D$ be a finite subgraph of the hexagonal lattice and let $\xi$ be a loop configuration on ${\rm Hex}$.
For $n,x>0$, the loop $O(n)$ model on $D$ with boundary conditions $\xi$ and edge-weight $x>0$ is 
the probability measure supported on loop configurations that coincide with $\xi$ on $D^{c}$ and defined by:
\[
	{\sf Loop}_{D,n,x}^{\xi}(\omega) = \frac{1}{Z^{\xi}(D,n,x)}\cdot n^{\# \text{loops}}x^{\# \text{edges}},
\]	
where $\# \text{edges}$ stands for the number of edges in $\omega \cap D$
and ${\# \text{loops}}$ for the number of finite loops of $\omega$ that intersect $D$.
Note that the number of bi-infinite paths intersecting $D$ is the same for all $\omega$ in the support ${\sf Loop}_{D,n,x}^{\xi}$, 
and we omit it from the count for convenience.

Gibbs measures for the loop $O(n)$ model are defined via the DLR condition in a straightforward fashion. 
As for FK-percolation, it is not immediate that thermodynamic limits are Gibbs.
Moreover, it is \emph{a priori} not clear that any Gibbs measures exist. 
The following result settles these questions in the regime of positive association for the spin representation of the model 
(see Section~\ref{sec:loop-background} for details on the spin representation). 

\begin{thm}\label{thm:loop}
	Let $n\geq 1$ and $0< x \leq \tfrac{1}{\sqrt{n}}$. 
	Then, there exists a unique Gibbs measure for the loop $O(n)$ model with edge-weight $x$.
	This measure is translation invariant and tail trivial.
	Moreover, it is the weak limit of ${\sf Loop}_{D_{k},n,x}^{\xi_{k}}$,
	for any sequence of domain $D_{k}$ increasing to ${\rm Hex}$ and any sequence $\xi_{k}$ of loop configurations.
\end{thm}

	When $n=1, x\in [0,1]$, the loop $O(n)$ model is in correspondence with the ferromagnetic Ising model on the triangular lattice and the statement is known~\cite{Aiz80,Hig81}.
	
	The statement was recently proven also in the following two ranges of parameters: when $n\in [1,2]$ and $x=x_{c}(n) := \tfrac{1}{\sqrt{2+\sqrt{2-n}}}$~\cite{DumGlaPel21}; when $n\in [1,1+\eps)$ and $x\in (\tfrac{1}{\sqrt{n}}-\eps,\tfrac{1}{\sqrt{n}}]$. In fact a stronger statement was shown there: strong Russo--Seymour--Welsh (RSW) estimates. 
	Such estimates are expected to hold for all $n\in [0,2]$ and $x \geq x_{c}(n)$.
	
	The authors of the present paper proved~\cite{GlaMan18} the uniqueness of the Gibbs measure at $n=2$ and $x=1$, a result which is not covered by Theorem~\ref{thm:loop}.
	However, it is still unproved in this case that the weak limit of any converging sequence of finite volume measures is Gibbs.	 

When $n$ is large, the behaviour is different and is better understood through the Pirogov--Sinai theory and similar techniques.
More precisely, for some $c,C >0$, if $n > C$ and $nx^6 <c$, the work of Dobrushin--Shlosman~\cite{DobShl85} implies the existence of a unique Gibbs measure, whereas when both $n$ and $nx^6$ are larger than $C$,~\cite{DumPelSamSpi17} proves the existence of at least three distinct extremal Gibbs measures --- these approximate three different ground states and are translations of each other.
Moreover, it is likely that~\cite[Thm~1.8]{DumPelSamSpi17} may be combined with the technique of~\cite{DobShl85}
to prove that these are the only extremal Gibbs measures for the loop $O(n)$ model for parameters as above.

Further background on the loop $O(n)$ model is given in Section~\ref{sec:loop-background}, while the proof of Theorem~\ref{thm:loop} is discussed in Section~\ref{sec:loop_proof}.
This is achieved by adapting the proof of Theorem~\ref{thm:mainFK} to the spin representation of the loop $O(n)$, then transferring the result to the actual loop model.

\paragraph{Organisation of the paper.}
In Section~\ref{sec:background} we give a quick overview of the properties of the three models of interest: FK-percolation, the Potts and the loop $O(n)$ models. 
The following two sections deal with the proof of the main result for FK-percolation: Section~\ref{sec:FK_preparation} sets up the context for the duplication argument, 
which is then performed in Section~\ref{sec:FK_duplication}. 
Corollary~\ref{cor:every-limit-is-gibbs} is proved in Section~\ref{sec:every-limit-is-gibbs}. 

Section~\ref{sec:Potts_proof} contains the proof of the main result for the Potts model, namely Theorem~\ref{thm:mainPotts}. 
Finally, Section~\ref{sec:extensions} contains two extensions of Theorem~\ref{thm:mainFK}: the same result for the triangular and hexagonal lattices and the proof of Theorem~\ref{thm:loop}.

\paragraph{Acknowledgements}
The authors would like to thank Yinon Spinka for sharing his ideas regarding the exploration argument 
and Sebastien Ott for helpful discussions and background.

A significant part of this work was conducted in Fribourg, and the first author is grateful to the University of Fribourg for its hospitality and to the Swiss National Science Foundation for financing these visits.
The second author is supported by the Swiss National Science Foundation and is a member of the NCCR SwissMAP.

\section{Background}\label{sec:background}

For the rest of the paper we will work with the lattice $\bbL = (V(\bbL),E(\bbL))$ whose vertices are 
$\{(x,y) \in \bbZ^2: x+y \text{ even}\}$ and which contains edges between vertices at euclidian distance $\sqrt 2$.

Write $\bbH$ for the half-plane $\bbR \times [0,+\infty)$ and $\bbH^-$ for its vertical reflection $\bbR \times (-\infty,0]$.
For $n \geq 1$, let the box of size $n$ be defined by $\La_n := [-n,n]^2$. 
We will routinely identify these regions to the graphs spanned by the edges of $\bbL$ contained in them.  

Everywhere in this work we assume $q \geq 1$ and $p \in (0,1)$.

\subsection{FK-percolation}\label{sec:fk-background}

We will use standard properties of FK-percolation. 
They can be found in~\cite[Chapters 3 and 4]{Gri06}, and we only recall them briefly below. 
Fix a finite subgraph~$D = (V,E)$ of~$\bbL$.

\paragraph{Monotonic properties.} 
The partial order on $\{0,1\}^{E(\bbL)}$ is defined by setting that $\omega\geq \omega'$ if $\omega_{e}\geq \omega'_{e}$, for all $e\in E(\bbZ^{2})$.
An event~$A$ is called {\em increasing} if for any~$\omega\ge\omega'$, $\omega'\in A$  implies that~$\omega\in A$.
For two boundary conditions $\xi$ and $\xi'$ on $D$ we write $\xi\ge\xi'$ if any two boundary vertices wired together in $\xi'$ are also wired together in $\xi$. 

Fix~$q\ge1$,~$p \in (0,1)$, and boundary conditions~$\xi\ge\xi'$. Then, for every  increasing events~$A$ and~$B$,
\begin{align*}
	\phi_{D,p,q}^\xi[A\cap B]&\ge \phi_{D,p,q}^\xi[A]\, \phi_{D,p,q}^{\xi}[B],  \tag{PosAssoc}\label{eq:FKG} \\
	\phi_{D,p,q}^{\xi}[A]&\ge \phi_{D,p,q}^{\xi'}[A].   	\tag{CBC}\label{eq:CBC} 
\end{align*}
The inequalities above will respectively be referred to as the {\em positive association} and the {\em comparison between boundary conditions}. 
The second inequality may also be stated as $\phi_{D,p,q}^{\xi}$ dominates stochastically $\phi_{D,p,q}^{\xi'}$, which we write 
$\phi_{D,p,q}^{\xi}\ge_{\rm st} \phi_{D,p,q}^{\xi'}$. 
We direct the reader to~\cite[Ch.~4]{Gri10} for details on stochastic domination.

\paragraph{Spatial Markov property.} 
Let $D'$ be a subgraph of $D$ and let $\zeta$ be a percolation configuration on $\bbZ^2$.
As discussed above (Section \ref{sec:intro-fk}), $\zeta$ defines boundary conditions on $D$ and $D'$~--- denote them by $\xi$ and $\xi'$, respectively. Then,
\begin{equation}\tag{SMP}\label{eq:SMP} 
	\phi_{D,p,q}^{\xi}[\, \cdot \, \,|\,\omega_e=\zeta_e,\forall e\notin D'] = \phi_{D',p,q}^{\xi'}[\,\cdot\,].
\end{equation}

\paragraph{Finite energy property.} 
There exists $\eps = \eps (p,q) > 0$ such that the following holds. 
For any subgraph $D = (V,E)$ of $\bbZ^{2}$, any boundary conditions $\xi$, any $E' \subset E$, and any percolation configuration $\omega_0$,
\[
	\phi^{\xi}_{D,p,q}[\omega = \omega_0 \text{ on $E'$}] \geq \eps^{|E'|}.
\] 

\paragraph{Infinite-volume limits.} 
Due to~\eqref{eq:CBC} and~\eqref{eq:SMP}, for any growing sequence of subgraphs $D_{n}$ increasing to $\bbZ^{2}$, 
the measures  $\phi_{D_{n},p,q}^0$ form a stochastically increasing sequence. 
Hence their weak limit exists and is independent of the sequence $D_{n}$. 
The limit is called the {\em free infinite-volume measure for FK-percolation} and is denoted by $\phi_{p,q}^0$.
Similarly, the sequence $\phi_{D_{n},p,q}^1$ is stochastically decreasing, its weak limit is called the {\em wired infinite-volume measure for FK-percolation} and is denoted by $\phi_{p,q}^1$.
The measures $\phi_{p,q}^0$ and $\phi_{p,q}^1$ are invariant under translations and ergodic, and $\phi^0_{p,q} \leq_{\rm st}\phi^1_{p,q}$. 
Depending on the values of $p$ and $q$, these two measures may be equal or not.

\paragraph{Dual model.} 
Define the dual lattice of $\bbL$ as $\bbL^* = \bbL + (1,0)$, that is its translate to the right by one unit. 
Notice that each vertex of $\bbL^*$ corresponds to a face of $\bbL$ 
and each edge $e$ of $\bbL$ is crossed by a unique edge $e^*$ of $\bbL^*$. 
We will call the lattice $\bbL$ and its edges {\em primal}, while those of $\bbL^*$ are called {\em dual}. 
For a subgraph $D = (V,E)$ of $\bbL$, the dual graph~$D^*=(V^*,E^*)$ of~$D$ is the subgraph of $\bbL^*$  spanned by the edges dual to those of~$D$.

Given a subgraph configuration~$\omega$, construct its dual configuration~$\omega^*$ on~$D^*$ by
declaring any edge of the dual graph to be open (resp.\ closed) if the
corresponding edge of the primal lattice is closed (resp.\ open) for $\omega$. 
Connections in $\omega^*$ are written $\xlra{*}$ or $\xlra{\omega^*}$. 

When $\omega$ is sampled according to $\phi_{D,p,q}^\xi$ for some boundary conditions $\xi$, 
its dual configuration $\omega^*$ follows an FK-percolation distribution
with the same parameter~$q$, dual parameter~$p^*$ satisfying~
\begin{align}
	\frac{p^*p}{(1-p^*)(1-p)}=q, \label{eq:pp*}
\end{align}
and dual boundary conditions $\xi^*$. 

We refer to~\cite{Gri06} for a full description of dual boundary conditions
and simply mention that if $D$ is simply connected and if $\xi$ is induced by a configuration $\omega$ on $\bbL \setminus D$ 
with no infinite cluster, then $\xi^*$ are the boundary conditions induced by $\omega^*$. 
In particular, the dual of free boundary conditions are the wired ones. 
As a consequence of this last fact, if $\omega$ is sampled according to the infinite-volume measure $\phi^0_{p,q}$, then $\omega^*$ has law
$\phi^1_{p^*,q}$.

For each fixed $q$, the unique value of $p$ for which $p = p^*$ is called the {\em self-dual} point. 
It may be derived from~\eqref{eq:pp*} and was shown to be the critical point in~\cite{BefDum12b}, so that
\begin{align*}
	p_{c}(q) = p_{\rm sd}(q) = \frac{\sqrt q}{1 + \sqrt q}.
\end{align*}
When $p\neq p_{c}(q)$, by~\cite{BefDum12b}, the large-scale behaviour of the model is trivial and $\phi_{p,q}^{0} = \phi_{p,q}^{1}$.

\paragraph{Dichotomy at the self-dual point.}
It was proved in~\cite{DumSidTas17} (see also~\cite{DumTas19} for a more general proof) that when $\phi_{p_{c}(q),q}^{0} \neq \phi_{p_{c}(q),q}^{1}$, the lengths of interfaces between clusters have exponential tails: for some $c>0$
\begin{align}\label{eq:exp_decay}
	\phi_{p,q}^0[0\lra \partial \La_n] \leq e^{-cn} \quad \text{ and }\quad 
	\phi_{p,q}^1[(1,0)\xlra{*}\partial \La_n] \leq e^{-cn} 
	\qquad \text{ for all $n\geq 1$}.
\end{align}
Moreover, it was proved in ~\cite{DumSidTas17} that $\phi_{p_{c}(q),q}^{0}=\phi_{p_{c}(q),q}^{1}$ when $1 \leq q \leq 4$. Conversely, when $q>4$ it was proved in~\cite{DumGagHar16} (see also~\cite{RaySpi20} for a shorter proof) that $\phi_{p_{c}(q),q}^{0} \neq \phi_{p_{c}(q),q}^{1}$.

\paragraph{Structure of Gibbs measures.}
It is well-known (see eg.~\cite[Theorem 7.26]{Geo11}) that every Gibbs measure is a linear combination of extremal Gibbs measures.
Thus, in order to describe the set of Gibbs measures of the model, it is enough to characterise its extremal measures, 
which are exactly those that are tail trivial~\cite[Thm. 7.7]{Geo11}.
It will be crucial for our arguments that tail trivial Gibbs measures are positively associated (as opposed to other Gibbs measures, which may lack this property).
This fact is well-known but we include the proof for completeness.

\begin{lem}\label{lem:extremal-fkg}
	Let $\phi$ be a tail-trivial measure for FK-percolation with parameters $p\in (0,1)$ and $q\geq 1$.
	Then $\phi$ is positively associated, that is for any two increasing events $A,B\in \{0,1\}^{E(\bbL)}$, one has~\eqref{eq:FKG}.
\end{lem}

\begin{proof}
	For $\omega$ sampled from $\phi$, denote by $\xi_{n}$ the boundary conditions induced by $\omega$ on $\La_{n}$.
	Fix two increasing events $A,B$.
	Then, by \eqref{eq:DLR-FK} and \eqref{eq:FKG}, $\phi$-a.s., for any $n>0$,
	\[
		\phi [A\cap B \, | \, \xi_{n}] \geq \phi [A \, | \, \xi_{n}]\, \phi [B \, | \, \xi_{n}].
	\]
	When regarded as functions of $n$, each probability above is a bounded backward martingale, and therefore converges. 
	The limit is given by the probability of the corresponding event under $\phi [. \, | \, \calF_\infty]$, where $\calF_\infty$ is the tail sigma-algebra. 
	By tail triviality, all three limits are a.s. constant, and equal to the unconditional probabilities of the corresponding events.
\end{proof}

\paragraph{A bound on the number of large clusters.} 
The following result is a direct consequence of~\eqref{eq:exp_decay}, and will come in handy at different times. 

\begin{lem}\label{lem:Kcross}
	Let $p\in (0,1)$ and $q\geq 1$ be such that $\phi_{p,q}^{0} \neq \phi_{p,q}^{1}$.
	There exists $K \geq 1$ and $c > 0$ such that,
	for any $n$ and any boundary conditions $\xi$ on $\La_{2n}$, 
	\begin{align}\label{eq:maxK}
		\phi_{\La_{2n},p,q}^\xi[\exists K \text{ disjoint clusters intersecting both $\La_n$ and $\partial \La_{2n}$}] \leq e^{-cn}.
	\end{align}
\end{lem}

\begin{proof}
	We start by proving a version of the statement for $\xi = 0$ and when all edges between $\La_{n-1}$ and $\partial \La_n$ are conditioned to be closed. Write $\partial \La_n \equiv 0$ for the latter event.
	By~\eqref{eq:exp_decay}, for some $c'>0$,
	\[
		\phi^{0}_{\La_{2n}, p,q}[\exists \text{ a crossing from $\La_n$ to $\partial\La_{2n}$} \, | \, \partial \La_n \equiv 0] \leq e^{-c'n}.
	\]
	Order the points on $\partial\La_{n}$ in an arbitrary order and explore the primal clusters containing them, one after the other.
	Conditionally on $\partial \La_n \equiv 0$ and on the first $k$ explored clusters, 
	the measure in the unexplored part of $\La_{2n}\setminus \La_{n}$ has free boundary conditions. 
	Thus, applying~\eqref{eq:CBC} and the above bound, we find that for any $k\geq 0$,
	\begin{align*}
		&\phi^{0}_{\La_{2n}, p,q}[\exists \text{ $k+1$ disjoint clusters intersecting both $\La_n$ and $\partial \La_{2n}$}\, | \, \partial \La_n \equiv 0] \\
		&\qquad \leq e^{-c'n}\,\phi^{0}_{\La_{2n}, p,q}[\exists \text{ $k$ disjoint clusters intersecting both $\La_n$ and $\partial \La_{2n}$}\, | \, \partial \La_n \equiv 0].
	\end{align*}
	Apply this bound repeatedly to conclude that, for any $K\geq 1$, 
	\[			
		\phi^{0}_{\La_{2n}, p,q}[\exists \text{ $K$ disjoint clusters intersecting both $\La_n$ and $\partial \La_{2n}$}\, | \, \partial \La_n \equiv 0] \leq e^{-Kc'n}.
	\]

	We turn our attention to the actual proof of~\eqref{eq:maxK}.
	By the finite energy property, conditionally on the existence of 
	$K$ disjoint clusters crossing from $\La_n$ to $\partial \La_{2n}$, the event $\partial \La_n \equiv 0$ occurs with a probability of at least $\eps^n$, for some constant $\eps>0$
	independent of $n$ or $K$. Moreover, changing the boundary conditions from free to $\xi$ affects the measure on $\La_{2n}$ by a Radon-Nikodym derivative bounded between $C^{-n}$ and $C^{n}$, for some constant $C$ independent of $n$ or $K$. 
	Thus
	\[
		\phi^{\xi}_{\La_{2n}, p,q}[\exists \text{ $K$ disjoint clusters intersecting both $\La_n$ and $\partial \La_{2n}$}] \leq (C/\eps)^{n}e^{-Kc'n}.
	\]
	The statement then follows by choosing $K$ sufficiently large.
\end{proof}

\begin{rem}
	The proof of Theorem~\ref{thm:mainFK} presented below aims to be generic and to apply to a variety of models with the properties above;
	as illustrated by the loop $O(n)$ model. 
	We will not attempt to describe here the minimal assumption needed for our proof, 
	but mention that positive association, some form of self-duality and of the spatial Markov property are necessary. 
	The exponential decay of~\eqref{eq:exp_decay} is also used extensively, 
	but this is not limiting as the dichotomy theorem of~\cite{DumTas19} applies to models with the properties above
	and yields~\eqref{eq:exp_decay} whenever the Gibbs measure is not unique. 
\end{rem}

\subsection{Potts model: background and Edwards-Sokal coupling}\label{sec:ES}

The Potts model and FK-percolation are related via the Edwards--Sokal coupling~\cite{EdwSok88}.
We start by stating the coupling and then explain its implications.

Fix an integer $q \geq 2$. For $T > 0$, let $p(T) = 1 - e^{-1/T}$. 

\paragraph{Finite volume.}
For a finite domain $D = (V,E)$, 
we will define a probability measure $\Phi$ on $\{0,1\}^E \times \{1,\dots, q\}^V$ whose marginals are 
the FK-percolation with parameters $(p(T),q)$ and the Potts model with parameters $(T,q)$, respectively.
The measure under boundary conditions $\tau : \partial D \to \{0,1,\dots, q\}$ is defined by
\begin{align}\label{eq:PhiD}
    \Phi_D^\tau [(\omega,\sigma)] = \tfrac{1}{Z_D^{\tau}}\big(\tfrac{p}{1-p}\big)^{|\omega|}
    \1_{\{\sigma(u) = \tau(u): \, \forall u\in \partial D\text{ with }\tau(u) \neq 0\}}
    \1_{\{\sigma(u) = \sigma(v): \, \forall uv \in \omega\}},
\end{align}
for all $(\omega,\sigma) \in \{0,1\}^E \times \{1,\dots, q\}^V$, where $Z_D^{\tau}$ is a normalising constant. 
Call a pair of configuration $(\omega,\sigma)$ {\em compatible} if any open edge of $\omega$ connects vertices of equal spin in $\sigma$, 
which is to say that the second indicator above is equal to $1$.  
Call $\sigma$ compatible with $\tau$ if the first indicator above is equal to $1$. 
Then $\Phi_D^\xi$ is supported on the set of compatible pairs of configurations with $\sigma$ compatible with $\tau$. 

We start by describing why the marginal on $\sigma$ of $\Phi_D^\tau$ is a Potts measure. 
Fix a spin configuration $\sigma \in \{1,\dots,q\}^{V}$ coherent with the boundary conditions $\tau$.
Summing $ \Phi_D^\tau [(\omega,\sigma)] $ over all $\omega$ compatible with $\sigma$, 
we observe that every pair of adjacent $u,v\in D$ contributes $1$ if $\sigma_{u}\neq \sigma_{v}$ (since $\omega_{uv} = 0$) 
and $1+ \tfrac{p}{1-p}$ if $\sigma(u) =\sigma(v)$ (since $\omega_{uv}$ may take values $0$ or $1$).
Due to the choice of $p(T)$, $1+ \tfrac{p(T)}{1-p(T)}  = e^{1/T}$.
Hence
\begin{align*}
	\sum_\omega \Phi_D^\tau [(\omega,\sigma)] = \tfrac{1}{Z_D^{\tau}}\exp \big(\tfrac{1}{T}\cdot \sum_{u\sim v} \1_{\{\sigma_{u} = \sigma_{v}\}}\big),
\end{align*}
which coincides with the definition of the $q$-state Potts model given by~\eqref{eq:def-potts}.

As a consequence of the above, we also deduce that the law of $\omega$ conditionally on $\sigma$ is obtained by opening and closing edges independently, with edges $uv$ with $\sigma(u) =\sigma(v)$ having probability $p(T)$ to be open, while all other edges have probability $0$ to be open. \smallskip 

We now discuss the marginal on $\omega$.
Fix an edge configuration $\omega\in \{0,1\}^{E}$.
If there exist vertices $u,v \in \partial D$ which are connected in $\omega$ and such that $\tau(u)$ and $\tau(v)$ are distinct and both non-zero, then there exists no $\sigma$ which is coherent with both $\omega$ and $\tau$, and $\Phi_D^\tau[\omega] = 0$. 

Assume next that for any two $u,v \in\partial D$ which are connected by a path in $\omega$, 
either $\tau_{u}=\tau_{v}$, or $\tau_{u}= 0$, or $\tau_{v}= 0$.
The spin configurations $\sigma$ compatible with $\omega$ 
are obtained by attributing the same spin to all vertices in each connected component of $\omega$. 
Moreover, to ensure that $\sigma$ is coherent with $\tau$, any connected component of $\omega$ that contains a vertex $v \in \partial D$ with $\tau(v) \neq 0$ is necessarily attributed spin $\tau(v)$. 
Summing over all spin configurations $\sigma$ compatible with $\omega$ we find
\begin{align*}
	\sum_\sigma \Phi_D^\tau [(\omega,\sigma)] = \tfrac{1}{Z_D^{\tau}}\big(\tfrac{p}{1-p}\big)^{|\omega|}q^{k^{\tau}(\omega)},
\end{align*}
where $k^{\tau}(\omega)$ is the number of connected components of $\omega$ that contain no vertex $v \in \partial D$ with $\tau(v) \neq 0$.

In particular, if $\tau$ takes only values $0$ and at most one other value, say $1$, then 
$\omega$ has the law of FK-percolation on $D$ with wired boundary conditions on $\{u \in \partial D: \sigma(u) = 1\}$
and free on the rest of $\partial D$. 
When $\tau$ takes more than one non-zero value,  
$\omega$ has the law of FK-percolation on $D$ with wired boundary conditions on each of $\{u \in \partial D: \tau(u) = i\}$ for $i = 1,\dots,n$, 
conditioned that each of these parts of $\partial D$ are disconnected from each other. 

Finally, from the above it is immediate that, conditionally on $\omega$, $\sigma$ is uniform among configurations coherent with both $\omega$ and $\tau$.

\paragraph{Infinite volume.}
The Edwards--Sokal coupling in infinite volume may be constructed similarly to that in finite volume.
Let $\mu$ be a Gibbs measure for the Potts model on $\bbL$.
Define a measure $\Phi$ on pairs of edge and spin configurations $(\omega,\sigma)\in \{0,1\}^{E(\bbL)} \times \{1,\dots, q\}^{V(\bbL)}$
by sampling $\sigma$ according to $\mu$, then sampling $\omega$ by opening each edge $uv$ with probability $p(T)$ if $\sigma(u) =\sigma(v)$ and $0$ otherwise, independently.

Defined in this way, the measure $\Phi$ satisfies a version of the DLR condition: 
for any finite domain $D = (V,E)$, conditionally on $\omega$ on $E(\bbZ^2) \setminus E$ and on $\sigma$ on $V(\bbZ^2) \setminus V$, 
the restriction of $\Phi$ to $D$ is $\Phi_D^\tau$, 
where $\tau(u) = 0$ for vertices $u \in \partial D$ which are disconnected in $\omega \cap E^c$ from $V^c$ 
and $\tau(u) = i$ for all vertices $u \in \partial D$ that are connected in $\omega \cap E^c$ to a vertex $v \in V^c$ with $\sigma(v) = i$.
Note that for the conditioning to be non-degenerate, all vertices $v \in V^c$ connected in $\omega \cap E^c$ to some $u \in \partial D$ need to have same spin. 
Observe also that $\tau$ depends only on $\omega$ for the edges of $E^c$ incident to $\partial D$ and 
on $\sigma$ for the vertices of~$V^c$ neighbouring vertices of $\partial D$.

\paragraph{Sampling  Potts from FK.}
The marginal on $\omega$ of a measure $\Phi$ constructed as above is not {\em a priori} a Gibbs measure for FK-percolation.
However, this is indeed the case when $\Phi$ was constructed from one of the monochromatic Potts measures or from the free one.

It is indeed easy to see that $\mu_{T,q}^{i}$ leads to the wired FK-percolation measure $\phi^{1}_{p(T),q}$ for any $i \in \{1,\dots, q\}$.
Moreover, the procedure can be reversed: $\mu_{T,q}^{i}$ may be sampled from $\phi^{1}_{p(T),q}$ by colouring each finite cluster uniformly in one of the $q$ colours, independently from other clusters, and colouring the unique infinite cluster (if such exists) in colour $i$.
In this way, the tail triviality of $\mu_{T,q}^{i}$ may be deduced from that of $\phi^{1}_{p(T),q}$.

In addition, $\mu_{T,q}^{\mathrm{free}}$ can be sampled from $\phi^{0}_{p(T),q}$ by colouring every (finite or infinite) cluster uniformly and independently of all other clusters.
Thus, $\mu_{T,q}^{\mathrm{free}}$ is tail trivial whenever $\phi^{0}_{p(T),q}$ exhibits no infinite cluster, 
and it equals $\tfrac{1}{q}(\mu_{T,q}^{1} + \dots + \mu_{T,q}^{q})$ when $\phi^{0}_{p(T),q}$ exhibits a unique infinite cluster.

\subsection{Loop $O(n)$ model}\label{sec:loop-background}

It is straight-forward that the loop-measure ${\sf Loop}_{D,n,x}$ itself satisfies neither positive association nor the finite energy property due to hard-core constraints: no vertex can have degree one or three.
Fortunately, the spin representation of the loop $O(n)$ model that was introduced in~\cite{DumGlaPel21} (not to be confused with the spin $O(n)$ model) does satisfy these properties when $n\geq 1$ and $x\leq \tfrac{1}{\sqrt{n}}$.

\paragraph{Spin representation.}
Let $F({\rm Hex})$ denote the set of faces of the hexagonal lattice.
For a loop configuration $\omega$ on ${\rm Hex}$, call a spin configuration $\sigma \in \{-1,1\}^{F({\rm Hex})}$ coherent with $\omega$ if the spins at any two neighbouring faces differ if and only if the edge common to the two faces belongs to $\omega$.
This determines $\sigma$ up to a global spin flip and allows to define the push-forward of loop measures to spin configurations.
Under boundary conditions $\tau  \in \{-1,1\}^{F({\rm Hex})}$, it is supported on spin configurations $\sigma \in \{-1,1\}^{F({\rm Hex})}$ that agree with $\tau$ on faces outside of $D$ and may be written as:
\begin{equation}\label{eq:def-spin-rep}
	{\sf Spin}_{D,n,x}^{\tau}(\sigma) = \frac{1}{Z^{\tau}_{\sf spin}(D,n,x)}\cdot n^{k_{\pm}(\sigma)}x^{\#\{u \sim v \colon \sigma_{u}\neq \sigma_{v}\}},
\end{equation}
where $k_{\pm}(\sigma)$ is the total number of finite spin-clusters in $\sigma$ (both of plus and minus spins) that intersect $D$ 
and the exponent for $x$ is the number of pairs of adjacent faces with different spins at least one of which is in $D$.
The number of finite loops intervening in ${\sf Loop}_{D,n,x}^{\xi}(\omega)$ appears here as $k_{\pm}(\sigma)$, since each such loop is the external boundary of a finite spin-cluster intersecting $D$;  see~\cite{DumGlaPel21} for details on this correspondence.
As in the definition of ${\sf Loop}_{D,n,x}^{\xi}$,  $k_{\pm}(\sigma)$ may be modified to include the infinite spin-clusters that intersect $D$ with no effect on the definition of ${\sf Spin}_{D,n,x}^{\tau}$.

We note that, written in this way, the loop $O(n)$ model may be viewed as a mixture of the Ising model and FK-percolation. 
In particular, the {\em spatial Markov} and {\em finite energy} properties hold. 
Notice that the boundary conditions induced by $\tau$ depend only on the spins of $\tau$ on the faces neighbouring $D$
and on connectivities induced on the boundary faces when $\tau$ outside of $D$ is viewed as a face percolation configuration.

\paragraph{Monotonic properties.}
As for FK-percolation, define the partial order on $\sigma \in \{-1,1\}^{F({\rm Hex})}$ pointwise and say that an event $A \subset \{-1,1\}^{F({\rm Hex})}$ is increasing if $\1_{A}$ is an increasing function with respect to this order.
It was shown in~\cite{DumGlaPel21} that when $n\geq 1$ and $0<x\leq \tfrac{1}{\sqrt{n}}$, the measure ${\sf Spin}_{D,n,x}^{\tau}$ is positively associated and the measures are stochastically ordered with respect to boundary conditions. That is, for any increasing events $A,B$ and any boundary conditions $\tau \geq \tau'$,
\begin{align}\label{eq:On_mono}
	{\sf Spin}_{D,n,x}^{\tau}[A\cap B] \geq {\sf Spin}_{D,n,x}^{\tau}[A] \,{\sf Spin}_{D,n,x}^{\tau}[B] \quad \text{and} \quad 
	{\sf Spin}_{D,n,x}^{\tau}[A] \geq {\sf Spin}_{D,n,x}^{\tau'}[A].
\end{align}

\paragraph{Infinite-volume measures.}
The notion of Gibbs measures for both the loop model and for the spin representation may be defined using the DLR property. 
For the spin model, a choice needs to be made in the way infinite spin-clusters are counted.
We make an arbitrary choice, similarly to that for FK-percolation: all infinite plus spin-clusters are considered wired at infinity, while the infinite minus spin-clusters are counted separately. This choice will ultimately turn out to be irrelevant.

As in FK-percolation, the monotonicity properties imply the existence of the infinite-volume limits with plus and minus boundary conditions
and that these are extremal Gibbs measures. We denote them by ${\sf Spin}_{n,x}^{+}$ and ${\sf Spin}_{n,x}^{-}$, respectively.
If these two measures coincide, then the Gibbs measure for the spin representation of the loop $O(n)$ model  (and hence for the model itself) is unique.

\paragraph{Dichotomy.}
Using the monotonicity properties~\eqref{eq:On_mono}, a dichotomy result for the loop~$O(n)$ model similar to that of FK-percolation~\cite{DumSidTas17} was proved in ~\cite{DumGlaPel21}:
for any $n\geq 1$ and $x\leq \tfrac{1}{\sqrt{n}}$, the model either satisfies RSW estimates and exhibits {\em macroscopic loops} or 
it exhibits {\em exponential decay} of loop lengths.
The dichotomy result is expected to extend to all $n \geq 0$ and $x > 0$.
Moreover, the macroscopic loops regime is conjectured~\cite{Nie82,BloNie89} to hold for $n\in [0,2]$ and $x\geq x_{c}(n):= \tfrac{1}{\sqrt{2+\sqrt{2-n}}}$, and the model is then expected to exhibit a conformally invariant scaling limit~\cite{KagNie04,Smi06}. 
For all other parameters, the model is expected to be in the exponential decay regime. 

At the moment, the phase diagram is understood only in several regions of parameters.
For recent results on the two types of behaviour, we direct the reader to the recent works~\cite{CraGlaHarPel20,DumGlaPel21, GlaMan18, GlaMan18b}.

\section{FK-percolation: preparation}\label{sec:FK_preparation}

\subsection{Case of translation invariant measures}\label{sec:translation_inv}

We call a measure $\phi$ on $\{0,1\}^{E(\bbL)}$ translation invariant if it is invariant under the translations by two non-collinear vectors. 

\begin{prop}\label{prop:translation_inv}
	Let $\phi$ be a Gibbs measure for FK-percolation. 
	Assume that it is translation invariant and tail trivial.
	Then 
	\begin{align*}
	\phi[0\lra \infty] = 0 \quad  \text{or} \quad  	\phi[(1,0)\xlra{*} \infty] = 0.
	\end{align*}
	As a consequence, $\phi = \phi_{p,q}^0$ or $\phi =  \phi_{p,q}^1$.
\end{prop}

\begin{proof}
    It is standard (see eg.~\cite[Theorem 14.15]{Geo11}) that translation invariance and extremality imply ergodicity.
    For every integer $N\geq 0$, the event that there exist exactly $N$ infinite clusters is translation invariant and hence has probability $0$ or $1$. 
    Due to the finite energy property, any two infinite clusters may be merged into a single one with positive probability.
    It follows that the number of infinite clusters of $\phi$ is either $0$, $1$ or infinity. 
    The case of infinitely many infinite clusters may be excluded by the classical argument of Burton and Keane~\cite{BurKea89}, 
    or more directly by Lemma~\ref{lem:Kcross} when $\phi_{p,q}^{0} \neq \phi_{p,q}^{1}$ and by the explicit unique infinite-volume measure when $\phi_{p,q}^{0}=\phi_{p,q}^{1}$.
    
    As the same argument applies to the dual model, we conclude that $\phi$ contains at most one infinite primal cluster and one infinite dual cluster. 
    Finally,~\cite[Thm.~1.5]{DumRaoTas19} excludes the coexistence of unique infinite primal and dual clusters.
\end{proof}

\subsection{No percolation in half-plane with free boundary conditions}

Denote by $\tau_{1/0}$ the percolation configuration on $\bbL$ in which all edges of the upper half-plane are open and all other edges are closed. 
For a finite domain $D$, the Dobrushin boundary conditions are those induced by $\tau_{1/0}$, 
and the measure with these boundary conditions is denoted by $\phi_{D,p,q}^{1/0}$.
The measures on the domains $[-N,N]\times [0,N]$ with Dobrushin boundary conditions form a decreasing sequence, 
and therefore converge weakly to a measure on configurations on $\bbH$, which we denote by $\phi_{\bbH, p,q}^{1/0}$.

The goal of this section is the following statement.

\begin{prop}\label{prop:no_hp_perco}
	Let $q\geq 1$. Then, $\phi_{\bbH, p_{c}(q), q}^{1/0}[0\lra \infty] =0$. 
\end{prop}

\begin{proof}
	Below we omit $q$ and $p_{c}(q)$ from the notation for brevity.
	
	By construction $\phi_{\bbH}^{1/0}$ is invariant under horizontal translations and is tail trivial. 
	Let $\phi$ be the increasing limit of $\phi_{\bbH}^{1/0}\circ \sigma_{N}$, where $\sigma_{N}$ is the vertical translation by $N$
	($\phi_{\bbH}^{1/0}\circ \sigma_{N}$ is the measure on the half-plane $\bbR \times [-N,\infty)$, with free boundary conditions on the boundary of said half-plane).
	Then $\phi$ is an infinite-volume Gibbs measure which is also invariant under vertical translations, and therefore invariant under all translations. 
	
	Now, due to Proposition~\ref{prop:translation_inv}, $\phi$ contains either no primal infinite cluster or no dual infinite cluster. 
	Let $\rho$ be the automorphism of $\{0,1\}^{E(\bbL)}$, where $\rho(\omega)$ is the configuration obtained
	by shifting $\omega^*$ by $(-1,0)$ and reflecting it with respect to the horizontal axis.

	Then $\phi\circ \rho$ is also a Gibbs measure and we claim that 
	\begin{equation}\label{eq:no-hp-perc-reflection}
		\phi  \leq_{\rm st}\phi\circ \rho.
	\end{equation}
	Indeed, by the finite energy property, the weak limit
	\[
		 \phi_{\bbR\times [-M, N]}^{1/0} := \lim_{K\to \infty} \phi_{[-K,K]\times [-M, N]}^{1/0}
	\]
	exists. Moreover, the measures $\phi_{\bbR\times [-M, N]}^{1/0}$ are decreasing in $N$ and increasing in $M$.  
	Finally, due to the monotonicity properties of $\phi_{[-K,K]\times [0, N]}^{1/0}$,
	\[
		\phi =\lim_{M\to\infty}\, \lim_{N\to\infty}\, \phi_{\bbR\times [-M, N]}^{1/0} \quad \text{and} \quad  
		\phi\circ \rho = \lim_{N\to\infty}\, \lim_{M\to\infty}\, \phi_{\bbR\times [-M, N]}^{1/0},
	\]
	and~\eqref{eq:no-hp-perc-reflection} follows readily. 

	Now, by the choice of $\rho$, we have $\phi [0\lra \infty] = (\phi\circ\rho) [(1,0) \xlra{*} \infty]$. 
	Furthermore, by~\eqref{eq:no-hp-perc-reflection}, $\phi [0\lra \infty] \leq (\phi\circ\rho) [0\lra \infty]$. 
	Finally, Proposition~\ref{prop:translation_inv} applied to $\phi\circ \rho$
	proves that at least one of the events $\{(1,0) \xlra{*} \infty\}$ and $\{0\lra\infty\}$ has probability $0$. 
	Thus,
	\begin{align*}
		\phi_{\bbH}^{1/0}[0\lra \infty] \leq \phi [0\lra \infty] =0.
	\end{align*} 
\end{proof}

\begin{cor}\label{cor:infinite_touch}
	For $q \geq 1$ and $p = p_c(q)$, let $\phi$ be a tail trivial Gibbs measure for FK-percolation.
	Let $\omega$ be a random percolation configuration sampled from $\phi$.
	Then there exists at most one infinite primal cluster in the restriction of $\omega$ to $\bbH$. 
	If the infinite cluster exists, then it a.s. contains infinitely many points of $\bbR \times \{0\}$.
	The same applies to dual clusters and to all other half-planes.
\end{cor}

\begin{proof}
	Write $\calE$ for the event that in $\bbH$, the origin is connected to infinity but not to $(\bbR\setminus \{0\}) \times \{0\}$.
	We first prove that	$\phi (\calE)=0$.
	
	Define the random domain $\calD \subset \bbH$ by removing from $\bbH$ 
	all primal clusters that intersect $(\bbR\setminus\{0\}) \times \{0\}$.
	In addition, for $n>0$ and a realisation $D$ of $\calD$, define $D_{n}:= D\cap \La_{n}$.
	Denote by $\xi_{n}$ the boundary condition on $D_{n}$ in which $\partial D_n \cap \partial \La_n$ is wired and the rest of $\partial D_n$ is free.
	Due to the construction of $\calD$, when $\calD = D$, the boundary conditions induced on $D_n$ by the configuration outside of $D_n$ dominate $\xi_n$. 
	
	Notice that if $\calE$ occurs, $0$ is contained in $D_n$ and is connected to $\partial \La_n$ inside $D_n$. 
	Thus
	\[
		\phi [\calE \, | \, \calD = D]
		\leq \phi_{D_{n}, p_{c}(q),q}^{\xi_{n}} [0 \lra \partial \La_{n}] 
		\leq \phi_{\bbH\cap \La_{n}, p_{c}(q),q}^{1/0} [0 \lra \partial \La_{n}].
	\]
	Letting $n$ tend to infinity, averaging over $D$ and applying Proposition~\ref{prop:no_hp_perco}, we obtain
	\[
		\phi[\calE] \leq	 \phi_{\bbH, p_{c}(q),q}^{1/0}[0 \lra \infty] = 0,
	\]
	as desired. 
	
	By the finite energy property, we deduce more generally that a.s. there exists no infinite cluster in $\bbH$ which intersects $\bbR \times \{0\}$ finitely many times. 
		In particular, there exists a.s. at most one dual infinite cluster in $\bbH$ intersecting $\bbR \times \{0\}$, 
	since any two such clusters would be separated by a primal infinite cluster which intersects $\bbR \times \{0\}$ finitely many times. 
	Moreover, the finite energy property states that if there exist with positive probability two infinite dual clusters in $\bbH$, 
	they also have positive probability to both intersect $\bbR \times \{0\}$. As a consequence, we conclude that $\bbH$ contains a.s. at most one infinite dual cluster. 
	
	The same argument may be applied to the primal model, and we conclude that $\bbH$ contains at most one infinite primal cluster and that, 
	if such a cluster exists, it necessarily intersects $\bbR \times \{0\}$ infinitely many times. 
\end{proof}

\subsection{Half-plane percolation in symmetric domains}
\newcommand{\Slit}{{\rm Slit}}

For even integers $M,N\geq 0$, denote by~$\Slit(M,N)$ the graph obtained by gluing $\bbH$ and $\bbH^-$ along the segment $[-M,N] \times \{0\}$.
More precisely, $\Slit(M,N)$ is obtained from $\bbL$ by splitting every vertex in $ \big((-\infty,-M) \cup (N,\infty)\big)\times \{0\}$ into two vertices: one is incident only to the edges in $\bbH$ and the other is incident to the edges in $\bbH^{-}$.
In particular, the edges of $\Slit(M,N)$ are in bijection with those of $\bbL$.
The boundary of $\Slit(M,N)$ is formed of $\partial \bbH\setminus ([-M,N] \times \{0\})$ and  $\partial \bbH^-\setminus ([-M,N] \times \{0\})$. For simplicity, we refer to these as $\partial \Slit(M,N)^+$ and  $\partial \Slit(M,N)^-$, respectively. 
When no ambiguity is possible, we omit~$M$ and~$N$ from the notation.

Write $\phi_{\Slit,p,q}^{1/0}$ for the measure on $\Slit$ 
with wired boundary conditions on~$\partial \Slit^+$ and free ones on  $\partial \Slit^-$. 
Such a measure may be constructed as a limit of measures in $\La_{R} \cap \Slit$, 
with the boundary conditions on $\partial \La_R$ becoming irrelevant as $R \to \infty$. 
See Lemma~\ref{lem:perc-hp-inf-domain} below for a proof of this fact and Figure~\ref{fig:slit} for an illustration.
There is a slight abuse of notation here, since $\Slit$ and $\La_R \cap \Slit$ are not formally subgraphs of $\bbL$.
We will allow it as the structure of these graphs is sufficiently clear and the theory of FK-percolation on general graphs is standard.

We emphasise that $\phi_{\Slit,p,q}^{1/0}$  is not symmetric with respect to primal/dual edges: 
the wired conditions are one step closer on both endpoints of $[-M,N]\times \{0\}$ than the free ones.

The main results of this section are the proposition below and its corollary. 

\begin{prop}\label{prop:hp_perco}
	Let $p\in (0,1)$ and $q\geq 1$. If $\phi_{p,q}^{0} \neq \phi_{p,q}^{1}$, 
	then there exists a constant $c = c(p,q) > 0$ such that for any $M,N\geq 0$ 
	\begin{align*}
		\phi_{\Slit,p,q}^{1/0}[0 \xlra{\bbH} \partial \Slit^+ ] > c.
	\end{align*}
\end{prop}

\begin{rem}
	That $\phi_{\Slit,p,q}^{1/0}[0 \xlra{} \partial \Slit^+]$ is uniformly positive
	is a simple consequence of duality and of the hypothesis of the proposition; 
	this appeared  already in~\cite[Lem~2.3]{GeoHig00}. 
	The innovation of the above is that the percolation occurs in the half-plane. 
\end{rem}

Before giving the proof of Propositon~\ref{prop:hp_perco}, let us state the following consequence.

\begin{cor}\label{cor:hp_perco}
	Let $\phi$ be a Gibbs measure for FK-percolation with $p\in (0,1)$ and $q\geq 1$ that satisfy $\phi_{p,q}^{0} \neq \phi_{p,q}^{1}$. 
	Then,
	\begin{align*}
		\phi[0\xlra{\bbH} \infty] > 0\quad \text{ or }\quad	\phi[(1,0)\xlra{\omega^* \cap \bbH} \infty] > 0.
	\end{align*}
\end{cor}

\begin{figure}
\begin{center}
\includegraphics[width = 0.8\textwidth]{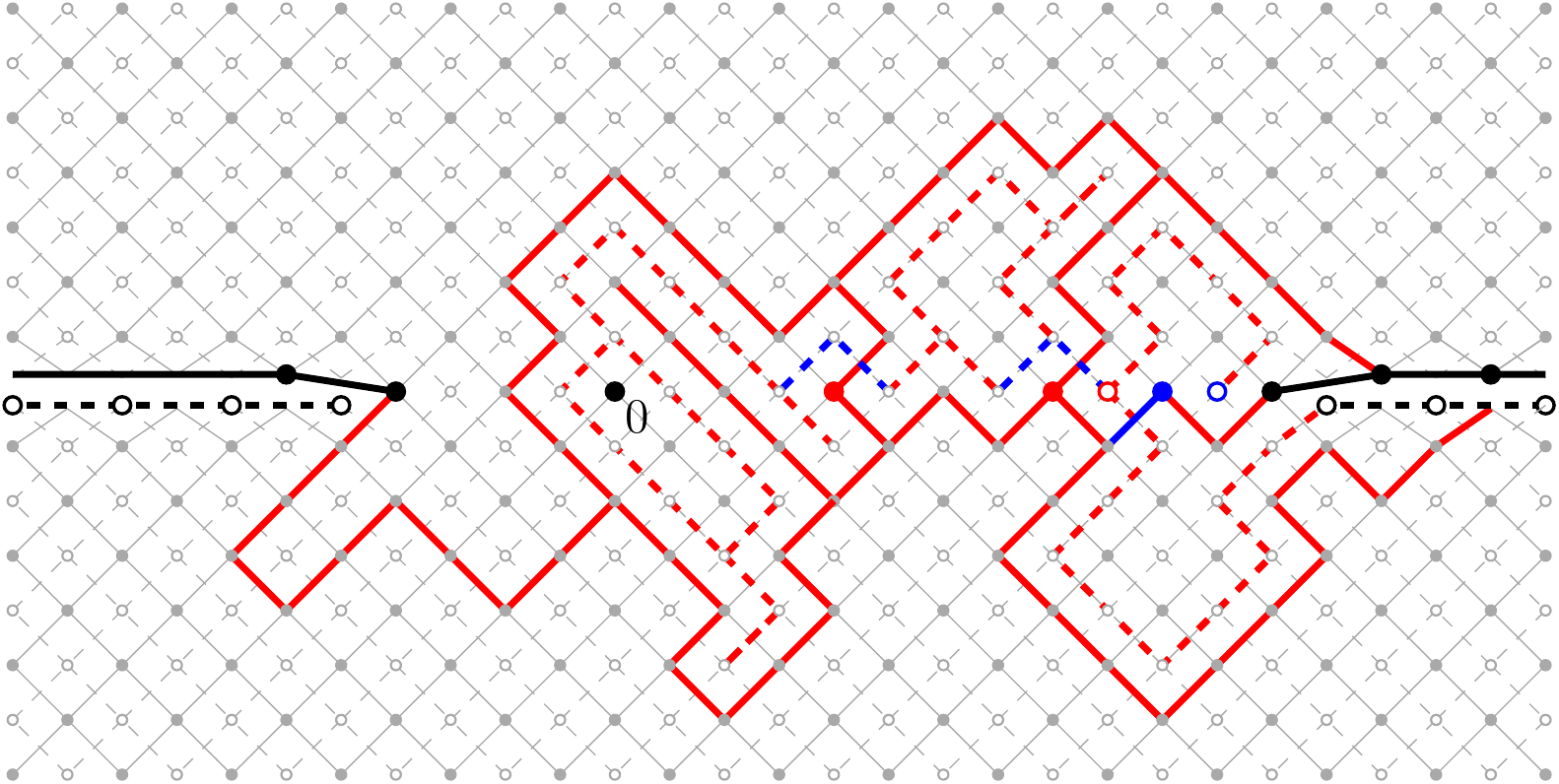}
\caption{The graph $\Slit$ with $\partial \Slit^+$ marked by solid black lines (wired) and  $\partial \Slit^-$ marked by dashed lines (free). 
For a configuration $\omega$ (in red), modifying the blue edges as indicated produces $\tilde \omega$. 
Notice that in $\tilde \omega$ the dual cluster of $(N-1,0)$ has diameter at least $N$ and does not intersect $\partial \Slit^-$.}
\label{fig:slit}
\end{center}
\end{figure}

Fix for the rest of this section $p\in (0,1), q\geq 1$ such that $\phi_{p,q}^{0} \neq \phi_{p,q}^{1}$ and omit them from the notation.

\begin{proof}[Corollary~\ref{cor:hp_perco}]
	Assume that $\phi[(1,0)\xlra{\omega^* \cap\bbH} \infty] =0$. 
	By the finite energy property, there exists a.s. no infinite dual cluster in $\bbH$. 
	Fix some $n>0$ and consider $K>0$ such that, 
	\[
		\phi [\La_{n}\xlra{\omega^* \cap\bbH} \La_{K}^{c}] <1/2.
	\]
	This means that, with probability at least $1/2$, there exists an open path in $\La_K$ that separates $\La_{n}$ from infinity.
	Let $\Gamma$ be the outermost such path and write $-M$ and $N$ for the first coordinates of the endpoints of $\Gamma$ on $\partial \bbH$.
	If no such path exists, write $\Gamma = \emptyset$. 
	By~\eqref{eq:FKG} and~\eqref{eq:CBC}, we have
	\begin{align*}
		\phi[0 \xlra{\bbH} \partial \La_n \, | \, \Gamma = \gamma] 
		\geq \phi^{1/0}_{\Slit(M,N)}[0 \xlra{\bbH} \partial \La_n] \geq c,
	\end{align*}
	for any realisation of $\gamma \neq \emptyset$ of $\Gamma$, 
	where $c >0$ is given by Proposition~\ref{prop:hp_perco} and is independent of $n$. 
	Averaging over $\gamma$ and using that $n$ is arbitrary, we find $\phi[0 \xlra{\bbH} \infty] \geq 	c /2 > 0$.
\end{proof}

The rest of the section is dedicated to proving Proposition~\ref{prop:hp_perco}. 
For $M,N\geq 0$ and $c > 0$, call a point~$z\in [-M,N]\times \{0\}$ $c$-{\em good} if 
\[
	\phi_{\Slit(M,N)}^{{1/0}} [ z \xlra{\bbH}  \partial \Slit(M,N)^+] \geq c.
\]
Our goal is then to show that, for some $c>0$ and every $M,N$, every point is $c$-good. 
The main difficulty stems from the fact that the measures $\phi_{\Slit(M,N)}^{{1/0}}$ for different $M,N$ are not stochastically ordered.
The proof is split into several steps: first we use the finite energy and local surgeries to show that there is a positive proportion of good points (Lemma~\ref{lem:perc-hp-some-good}); 
then we increase this proportion by changing $c$ (Lemma~\ref{lem:perc-hp-between-good}). 
Thus the proportion of $c$-good points may be rendered arbitrarily close to $1$ by choosing $c$ small enough (Lemma~\ref{lem:perc-hp-many-good}). 
Finally we use a counting argument to derive from the above that, for some $c>0$, every point is good.

We start with showing that $\phi_{\Slit}^{{1/0}}$ is well-defined. 

\begin{lem}[Dobrushin boundary conditions on $\Slit$]\label{lem:perc-hp-inf-domain}
	Fix $M,N\geq 0$ and write $\Slit = \Slit(M,N)$.
	Let $(\xi_{k})_{k \geq 1}$ be a sequence of boundary conditions on~$\La_k \cap \Slit$
	which are wired on~$\partial\bbH$ and free on~$\partial\bbH^-$. 
	Then the weak limit of $(\phi_{\La_k \cap \Slit}^{\xi_k})_{k\geq 1}$ exists and does not depend on the choice of~$(\xi_{k})_{k \geq 1}$.
\end{lem}

\begin{proof}
	For $k \geq \max(M,N)$, define the following four boundary conditions on $\La_k \cap \Slit$ 
	denoted by $\zeta_k^{0,0}$, $\zeta_k^{1,0}$, $\zeta_k^{0,1}$ and $\zeta_k^{1,1}$. 
	All of them are wired on $\partial\bbH$ and free on $\partial\bbH^{-}$. In addition 
	$\zeta_k^{0,0}$ and $\zeta_k^{0,1}$ are free on $\partial \La_k \cap \bbH$, while	$\zeta_k^{1,0}$ and $\zeta_k^{1,1}$ are wired on this part of the boundary;
	$\zeta_k^{0,0}$ and $\zeta_k^{1,0}$ are free on $\partial \La_k \cap \bbH^-$ and $\zeta_k^{0,1}$ and $\zeta_k^{1,1}$ are wired on $\partial \La_k \cap \bbH^-$.	Then 
	\begin{align*}
		\phi_{\La_k \cap \Slit}^{\zeta_k^{0,0}} \leq_{\rm st}	\phi_{\La_k \cap \Slit}^{\xi_k} \leq_{\rm st}	\phi_{\La_k \cap \Slit}^{\zeta_k^{1,1}}.
	\end{align*}
	Thus, the statement follows by proving that 
	\begin{align}\label{eq:slit1111}
		\lim_{k \to \infty}	\phi_{\La_k \cap \Slit}^{\zeta_k^{0,0}} = 	\lim_{k \to \infty}	\phi_{\La_k \cap \Slit}^{\zeta_k^{1,1}}.
	\end{align}
	Notice that the two weak limits exist, with the first being increasing, while the second is decreasing. 
	
	Fix $\eps > 0$, $r \geq  \max(M,N)$ and an increasing event $A$ depending only on the edges in $\La_r \cap \Slit$. 
	Then, due to~\eqref{eq:SMP} and~\eqref{eq:CBC}, 
	\begin{align*}
		\phi_{\La_k \cap \Slit}^{\zeta_k^{1,1}}[A] \leq \phi_{\La_k \cap \Slit}^{\zeta_k^{0,1}}[A \,|\, \La_r \nxlra{\omega^* \cap \bbH} \partial \La_k].
	\end{align*}
	Moreover, due to the finite energy property,
	\begin{align*}
		\phi_{\La_k \cap \Slit}^{\zeta_k^{0,1}}[\La_r \xlra{\omega^* \cap \bbH} \partial \La_k]
		\leq c_0 \phi_{\Lambda_{k}\cap \bbH}^{0/1}[(1,0) \xlra{\omega^* \cap \bbH} \partial \La_k]
	\end{align*}
	for some $c_0 = c_0(p,q,r,M,N)$. By  Proposition~\ref{prop:no_hp_perco}, the above may be rendered smaller than $\eps$ by taking $k$ large enough.
	Thus, we find 
	\begin{align*}
		\phi_{\La_k \cap \Slit}^{\zeta_k^{1,1}}[A] \leq  \phi_{\La_k \cap \Slit}^{\zeta_k^{0,1}}[A] + \eps
	\end{align*}
	for $k$ large enough. 	
	The same reasoning applies in $\bbH^-$, and we find 
	\begin{align*}
		\phi_{\La_k \cap \Slit}^{\zeta_k^{0,1}}[A] \leq  \phi_{\La_k \cap \Slit}^{\zeta_k^{0,0}}[A] + \eps
	\end{align*}
	for $k$ large enough. 
	
	Combining the last two displays and keeping in mind that $\eps > 0$ is arbitrary, we conclude that 
	\begin{align*}
		\lim_{k \to \infty}	\phi_{\La_k \cap \Slit}^{\zeta_k^{0,0}}[A] = 	\lim_{k \to \infty}	\phi_{\La_k \cap \Slit}^{\zeta_k^{1,1}}[A]
	\end{align*}
	for any increasing event $A$ depending on finitely many edges. Finally,~\eqref{eq:slit1111} may be deduced from the above by the monotone class theorem.
\end{proof}

\begin{lem}[Positive proportion of good points]\label{lem:perc-hp-some-good}
	There exist~$\alpha,c_0>0$ such that, for any~$M,N >0$, the number of~$c_0$-good points to the left (resp. right) of the origin in $\Slit(M,N)$ 
	is greater or equal than~$\alpha M$ (resp. $\alpha N$).
\end{lem}

\begin{proof}
	Fix small constants~$\alpha, c_0 >0$; we will see below how small they need to be. 
	Fix $M>0$. 
	It suffices to prove the property only for the points to the right of the origin and for $N$ sufficiently large. 
	We will proceed by contradiction and suppose that for some large $N$ the number of $c_0$-good points to the right of the origin is smaller than $\alpha N$.
	
	Recall the automorphism $\rho$ of $\{0,1\}^{E(\bbL)}$ 
	obtained by shifting the dual configuration by $(-1,0)$ and reflecting it with respect to the horizontal axis $\bbR\times \{0\}$. 
	Then, $\phi_{\Slit}^{{1/0}} \circ \rho$ is dominated by $\phi_{\Slit}^{{1/0}}$ \footnote{The domination, rather than equality, comes from the asymmetry in the boundary of $\Slit$.}. 
	Thus, if a point $(u,0)$ is not~$c_0$-good, 
	then the probability that $(u+1,0)$ is linked to $\partial \Slit^-$ in $\omega^* \cap \bbH^-$ is also smaller than~$c_0$. 
	As a consequence, 
	\begin{align*}
		\phi_{\Slit}^{{1/0}} \big[\#\{u\in [0,N] \cap 2\bbZ \colon (u,0) \xlra{\bbH} \partial \Slit^+ \text{ or } (u+1,0) \xlra{\omega^* \cap\bbH^-}  \partial \Slit^-  \} \big] 
		\leq \alpha N + c_0N.
	\end{align*}
	Define the following event:
	\[
		A:= \big\{\#\{u\in [0,N] \cap 2\bbZ \colon (u,0) \xlra{\bbH} \partial \Slit^+ \text{ or }  (u+1,0) \xlra{\omega^* \cap\bbH^-}  \partial \Slit^-  \} \leq 2(\alpha + c_0)N\big\}.
	\]
	Applying the Markov inequality, we find that 
	\begin{align}\label{eq:AAaaA}
			\phi_{\Slit}^{{1/0}} [A] \geq \tfrac12.
	\end{align}
	Since $\phi_{\Slit}^{{1/0}}$-a.s. $\bbH$ contains no dual infinite cluster (see Proposition~\ref{prop:no_hp_perco} or the proof of Lemma~\ref{lem:perc-hp-inf-domain}), that $(u,0)$ is connected to infinity in $\bbH$ implies a.s. that it is also connected to $\partial \Slit^+$ in $\bbH$. The same applies in $\bbH^-$ for dual connections. 
	 
	For~$\omega \in A$, let $\tilde \omega$ be the configuration
	obtained from $\omega$ by 
	opening the two edges below each dual vertex of $[0,N-1] \times \{0\}$ that is connected to $\partial \Slit^-$ in $\omega^* \cap \bbH^-$
	and closing the two edges above each primal vertex of $[0,N-1] \times \{0\}$ that is connected to $\partial \Slit^+$ in $\omega\cap \bbH$.
	See Figure~\ref{fig:slit} for an illustration. 
 
 	For any configuration $\tilde \omega$ obtained as above, the number of vertices appearing in~\eqref{eq:AAaaA} is zero. 
	Thus, the two following events occur for $\tilde\omega$
	\begin{align*}
		(-\infty,0) \times \{0\} \xlra{\tilde\omega^* \cap \bbH} (N - 1,0) \quad \text{and} \quad (-\infty,0]\times \{0\} \xlra{\tilde\omega \cap \bbH^{-}} (N,0).
	\end{align*}
	Write $\tilde A$ for the intersection of the events above. 

	Next we will lower bound the probability of $\tilde A$ in terms of that of $A$.  
	Observe that $\omega$ and $\tilde\omega$ differ on at most $4(\alpha + c_0)N$ edges. 
	Thus, by the finite energy property, there exists a constant $c_1 = c_1 (p,q) > 0$ independent of $c_0,\alpha$ or $N$ such that 
	\begin{align*}
		\phi[\tilde\omega] \geq \exp[-c_1 (\alpha + c_0)N]\,\phi[\omega] \qquad \forall \omega \in A.
	\end{align*}
	The function $\omega \mapsto \tilde\omega$ is not one-to-one, but the number of pre-images of any given $\tilde\omega$ is at most 
	\begin{align*}
		S_N(\alpha + c_0) = \sum_{k = 0}^{4(\alpha + c_0)N} \binom{N}{k}.
	\end{align*}
	Combining the two equations above, we find
	\begin{align}\label{eq:Alb}
		\phi[\tilde A] \geq \frac{e^{-c_1(\alpha + c_0)N}}{S_N(\alpha + c_0)} \phi[A] 
		\geq \frac{e^{-c_1(\alpha + c_0)N}}{2S_N(\alpha + c_0)}.
	\end{align}
	
	Split now $\tilde A$ into the event $\tilde A^0$, for which $(N-1,0)$ is connected to $\partial \Slit^-$ in $\tilde \omega^*$, 
	and its complement  $\tilde A^1$.
	When $\tilde A^1$ occurs, the dual cluster of $(N-1,0)$ has diameter at least $N$ and is surrounded by open edges. 
	A standard argument involving conditioning on the dual cluster of $\partial \Slit^-$ and using~\eqref{eq:SMP} and~\eqref{eq:CBC} yileds
	\begin{align}\label{eq:Aub1}
		\phi_{\Slit}^{{1/0}} [\tilde A^{1}] 
		\leq \phi^{1} [(0,1) \xlra{*} \partial \La_N ] 
		\leq e^{-c N},
	\end{align}
	with $c >0$ given by~\eqref{eq:exp_decay}.

	Assume now that $\tilde\omega \in \tilde A^0$ and close the two edges below $(N,0)$.
	In the new configuration, both points $(N-1,-1)$ and $(N+1,-1)$ are disconnected from $\partial \Slit^+$ 
	and the cluster of one of them has diameter at least~$N$.
	Using the finite energy property and the same argument as in \eqref{eq:Aub1}, we conclude that 
	\begin{align}\label{eq:Aub0}
		\phi_{\Slit}^{{1/0}} [\tilde A^{0}] 
		\leq c_2 \,\phi^{0} [0 \lra{} \partial \La_N ] 
		\leq c_2 \, e^{-c N},
	\end{align}
	with $c >0$ given by~\eqref{eq:exp_decay} and $c_2 = c_2(p,q)$ a constant independent of $N$. 
	
	Combining~\eqref{eq:Aub1} and~\eqref{eq:Aub0} using the union bound and introducing the result in~\eqref{eq:Alb}, 
	we conclude that 
	\begin{align*}
		S_N(\alpha + c_0) 	\geq \tfrac{1}{2(1+c_2)}\exp[ cN - c_1(\alpha + c_0)N].
	\end{align*}
	This is contradictory for $\alpha + c_0$ small enough and $N$ large enough. 	
\end{proof}

\begin{lem}[Increase proportion of good points]\label{lem:perc-hp-between-good}
	There exist~$\alpha, \beta >0$ such that, for any~$c>0$ and any even $-M \leq u \leq v \leq N$ such that $(u,0)$ and $(v,0)$ are both $c$-good, 
	there exist at least~$\alpha (v-u \vee 0)$ points $(w,0)$ with $u\vee 0 \leq w \leq v$ that are~$\beta c^2$-good.
\end{lem}

\begin{proof}
	Fix $c >0$ and $M$, $N$, $u$ and $v$ as above. 
	By the FKG inequality
	\[
		\phi_{\Slit}^{{1/0}} \big[(u,0) \xlra{\bbH} \partial \Slit^+ \text{ and }(v,0) \xlra{\bbH} \partial \Slit^+\big]\geq c^2.
	\]
	When the event above occurs, the measure in the domain bounded by the left-most crossing from~$u$ to $\partial \Slit^+$ in $\bbH$ 
	and the right-most crossing from~$v$ to $\partial \Slit^+$ in $\bbH$ 
	dominates~$\phi_{\mathrm{Slit}(u,v)}^{{1/0}}$.
	Therefore,  for each $u\vee 0 \leq w \leq v$ even, 
	\begin{align}
		&\phi_{\Slit}^{{1/0}} \big[(w,0) \xlra{\bbH} \partial \Slit^+\big]\nonumber\\
		&\qquad \geq\phi_{\Slit(-u,v)}^{{1/0}} \big[(w,0) \xlra{\bbH} \partial \Slit(-u,v)^+\big]
		\,\phi_{\Slit}^{{1/0}} \big[(u,0) \xlra{\bbH} \partial \Slit^+ \text{ and }(v,0) \xlra{\bbH} \partial \Slit^+\big]\nonumber\\
		&\qquad \geq c^2 \phi_{\Slit(-u,v)}^{{1/0}} \big[(w,0) \xlra{\bbH} \partial \Slit(-u,v)^+\big].\label{eq:SlitSlit}
	\end{align}
	Applying Lemma~\ref{lem:perc-hp-some-good} to $\Slit(-u,v)$ yields the result. 
	Formally, when $u \geq 0$, we apply Lemma~\ref{lem:perc-hp-some-good} to $\Slit(0,v-u)$.
\end{proof}

\begin{lem}[Large proportion of good points]\label{lem:perc-hp-many-good}
	For any~$\eps >0$, there exists~$c(\eps) >0$ such that, for any~$M,N \geq 0$ even, 
	the number of~$c$-good points to the right (resp. left) of the origin in $\Slit(M,N)$ 
	is greater or equal than~$(1-\eps)  N/2$ (resp. $(1-\eps) M/2$).
\end{lem}

Notice that in total there are $N/2$ primal points to the right of $0$. 
Thus, the above states that the proportion of $c$-good points may be rendered as close to $1$ as desired by choosing $c$ small enough. 

\begin{proof}
	Take~$\alpha,\beta,c_0>0$ such that both Lemmas~\ref{lem:perc-hp-some-good} and~\ref{lem:perc-hp-between-good} hold. 
	Fix~$M,N$. We estimate only the number of good points to the right of the origin. The number of good points to the left of the origin can be then estimated in the same way.
	
	By Lemma~\ref{lem:perc-hp-some-good} there are at least~$\alpha N$ points right of $0$ that are~$c_0$-good.
	By Lemma~\ref{lem:perc-hp-between-good}, between any two $c_0$-good points $(u,0)$ and $(v,0)$ right of $0$, 
	there exist at least~$\alpha (|u-v|-1)$ points that are~$\beta c_0^2$-good. 
	The same holds between $0$ and the left-most $c_0$-good point right of $0$. 
	Applying this for all pairs of consecutive~$c_0$-good points, 
	we deduce the existence of at most $(1-\alpha)^2 N$ points that are not~$\beta c_0^2$-good. 

	Iterating the procedure for $c_{k} = \beta c_{k-1}^2$, we find that there exist at most~$(1-\alpha)^kN$ points that are not~$c_k$-good. 
	The lemma follows by taking~$k$ large enough that~$(1-\alpha)^k<\eps$.
\end{proof}

\begin{proof}[Proposition~\ref{prop:hp_perco}]
	Fix~$\eps = 1/16$. Without loss of generality, assume that $M\geq N$.
	Take $c= c(\eps)>0$ such that the conclusion of Lemma~\ref{lem:perc-hp-many-good} holds.
	Denote by~$G_{-}$ and~$G_{+}$ the sets of first coordinates of $c$-good points in $[-M, -M+N/2]\times \{0\}$ 
	and in~$[N/2,N]\times \{0\}$, respectively, for $\Slit(M,N)$. 
	Then,
	\[
		|G_{-}|,|G_{+}| \geq \tfrac{1-\eps}4 N, \text{ whence }|G_{-}\times G_{+}| \geq \tfrac1{16} N^2(1- \eps)^2 
		\geq \tfrac1{32}N^2. 
	\]
	For any~$(u,v) \in G_{-}\times G_{+}$, we have~$v-u \in [M,M+N] \cap 2\bbZ$. 
	Thus, there exists at least one even integer~$L\in [M,M+N]$ such that
	there are at least $N/16$ pairs~$(u,v) \in G_{-}\times G_{+}$ which satisfy~$v - u = L$.
	Fix one such~$L$.
	Then the above implies existence of at least $N/16$ even integers $v\in [N/2,N]$ such that $(v-L,v) \in G_{-}\times G_{+}$.
		 
	Consider now the set $G_L$ of first coordinates of $c$-good points to the right of $0$ in $\Slit(L - N,N)$.
	By Lemma~\ref{lem:perc-hp-many-good}, $|G_L| \geq (1 - \eps)N$.
	Thus, there exists at least one $v \in [N/2,N]$ with $N-v \in G_L$ and $(v - L,v) \in G_{-}\times G_{+}$.
	Indeed, the sum of the cardinalities of these two sets is strictly larger than $N/2$. 
	
	Fix $v$ as above and write $u = v-L$ and $\Slit = \Slit(M,N)$. Then
	\begin{align*}
		\phi_{\Slit}^{{1/0}} \big[0 \xlra{\bbH} \partial\Slit^+\big] 
		&\geq c^2 \,\phi_{\Slit}^{{1/0}}\big[0 \xlra{\bbH} \partial\Slit^+ \big| 
					\,(u,0) \xlra{\bbH} \partial\Slit^+,\, (v,0) \xlra{\bbH} \partial\Slit^+\big] \\
	&\geq c^2 \,\phi_{\Slit(-u,v)}^{{1/0}} \big[0 \xlra{\bbH} \partial\Slit(-u,v)^+\big] \\
		&= c^2 \,\phi_{\Slit(L - N,N)}^{{1/0}} \big[(N-v,0) \xlra{\bbH} \partial\Slit(L - N,N)^+\big] 
		\geq c^3.
	\end{align*}
	The first inequality uses~\eqref{eq:FKG} and the fact that $(u,0)$ and $(v,0)$ are $c$-good for $\Slit(M,N)$; 
	the second one uses the same exploration argument as the one leading to~\eqref{eq:SlitSlit};
	the equality is obtained after shift by $N-v$;
	the third inequality uses that $(N-v,0)$ is $c$-good for $\Slit(L-N,N)$. 
	This concludes the proof. 
\end{proof}

\section{FK-percolation: duplication argument}\label{sec:FK_duplication}

We are finally ready to prove Theorem~\ref{thm:mainFK}. To do so, we will prove that all extremal Gibbs measures are translation invariant. 

\begin{prop}\label{prop:Gibbs=trans_inv}
	Let $\phi$ be an extremal Gibbs measure for FK-percolation. Then $\phi$ is translation invariant. 
\end{prop}

The above, combined with Proposition~\ref{prop:translation_inv}, immediately implies  Theorem~\ref{thm:mainFK}.

\begin{proof}[Theorem~\ref{thm:mainFK}]
	Let $\phi$ be an extremal Gibbs measure for FK-percolation. By Proposition~\ref{prop:Gibbs=trans_inv}, $\phi$ is translation invariant 
	and due to Proposition~\ref{prop:translation_inv}, $\phi$ is equal to $\phi^0$ or $\phi^1$.
	Since every Gibbs measure is a linear combination of extremal Gibbs measures, this yields the result.
\end{proof}

The rest of the section is dedicated to proving Proposition~\ref{prop:Gibbs=trans_inv}. 
The statement is obvious when $\phi_{p,q}^0=\phi_{p,q}^1$, and we assume henceforth that $p$ and $q$ are such that $\phi_{p,q}^0\neq \phi_{p,q}^1$. 

Write $\phi'$ for a translate of $\phi$.
To prove Proposition~\ref{prop:Gibbs=trans_inv}, we will show that $\phi = \phi'$. 
By symmetry it suffices to show that $\phi \leq_{\rm st} \phi'$, which we do next. 

Let $\bbP$ be the probability measure governing two independent samples $\omega$ and $\omega'$ with laws $\phi$ and $\phi'$, respectively. 
The goal will be to create arbitrarily large random domains $\calD$ measurable in terms of $\omega$ and $\omega'$ on $\calD^c$
such that the boundary conditions induced by $\omega'$ on $\partial \calD$ dominate those induced by $\omega$.

The main difference between our argument and that of~\cite{GeoHig00} appears below, in the way that $\calD$ is explored. 
Indeed, in~\cite{GeoHig00} (following the previous argument of ~\cite{Aiz80}), 
it is shown that $\bbP$-a.s. there exist infinitely many circuits around $0$ on which $\omega \leq \omega'$. 
Then one may take $\calD$ to be the domain delimited by the outer-most such circuit contained in a large box $\La_N$. 
For the Ising model, this suffices to conclude that the marginal for $\omega$ in $\calD$ is dominated by that for $\omega'$. 
For FK-percolation, boundary conditions $\xi$ and $\xi'$ induced by $\omega$ and $\omega'$, respectively, are not determined solely by the configurations on $\partial \calD$,
but also by the connections between points of $\partial \calD$ outside of $\calD$.
Thus, with the choice of $\calD$ as above, we may not conclude that $\xi \leq \xi'$. 

To solve this issue, we will explore $\calD$ so that its boundary is formed of one, two or four arcs, 
each of which is either open in $\omega'$ or closed in $\omega$.
Unfortunately, the exploration procedure for such a domain is considerably more complicated.

\subsection{Shielding arcs}

As proved in Corollary~\ref{cor:hp_perco}, each half-plane contains at least one infinite cluster, primal or dual. 
Corollary~\ref{cor:infinite_touch} implies that such a cluster is unique (if exists).
In light of the above, and of the tail triviality of $\phi$, each half-plane (below we will only refer to $\bbH$ and $\bbH^-$) 
is in one of the settings below: 
\begin{itemize}
	\item[(a)] contains a unique infinite primal cluster and no infinite dual cluster, $\phi$-a.s.;
	\item[(b)] contains a unique infinite dual cluster and no infinite primal cluster, $\phi$-a.s.;
	\item[(c-d)] contains both a unique infinite primal cluster and a unique infinite  dual cluster, $\phi$-a.s..
\end{itemize}
In the third case, recall from Corollary~\ref{cor:infinite_touch} that each infinite primal cluster as well as each infinite dual cluster 
intersects the horizontal axis infinitely many times. Thus, this case may further be split into two sub-cases: 
(c) the primal cluster intersects $(-\infty,0]\times \{0\}$ infinitely many times, while the infinite dual cluster touches  $[0,\infty)\times \{0\}$ infinitely many times, or (d) vice-versa. 
We then say that we have coexistence in $\bbH$ with the primal cluster to the left and the dual one to the right, or vice-versa. 

Our goal now is to create arcs $\calS$ that separate $0$ from $\infty$ in $\bbH$, 
which are explorable from above, and for which the boundary conditions induced by $\omega'$ dominate those induced by $\omega$. 
These paths will be called {\em shielding arcs}; a definition is given below and depends on the setting of $\phi$ in $\bbH$.

Consider a pair of configurations $(\omega, \omega')$.
For $n \geq0$ and $x <0$ and $y > 0$ even,
we say that a path $\varsigma$ in $\bbH$ is a shielding arc above $\La_n$ between $x$ and $y$ 
if it does not intersect $\La_n$ and one of the following holds:
\begin{itemize}
	\item[(a)] $\varsigma$ runs from $(x,0)$ to $(y,0)$ and is formed entirely of primal edges contained in~$\omega'$;
	\item[(b)] $\varsigma$ runs from $(x+1,0)$ to $(y+1,0)$ and is formed entirely of dual edges contained in~$\omega^*$;
	\item[(c)] $\varsigma$ runs from $(x,0)$ to $(y+1,0)$ and is 
		formed of an arc $\varsigma^{(1)}$ of primal edges all of which are contained in $\omega'$,
		followed by a single edge of $\bbZ^2$, followed by an arc $\varsigma^{(2)}$ of dual edges, all contained in $\omega^*$.
		Moreover, $\varsigma^{(1)}$ is connected to infinity in $\omega' \cap \bbH$ by a path that does not intersect~$\varsigma^{(2)}$;
	\item[(d)] $\varsigma$ runs from $(x+1,0)$ to $(y,0)$ and is 
		formed of a sequence of edges of $\omega^*$, followed by a single edge of $\bbZ^2$, followed by edges of $\omega'$.
		Moreover, the primal edges of $\varsigma$ are connected to $\infty$ in $\omega' \cap \bbH$ by a path that does not intersect the dual edges of~$\varsigma$.
\end{itemize}
In all cases, write ${\rm Above}(\varsigma)$ 
for the subgraph spanned by the edges of $\bbH$ separated from $[x,y] \times\{0\}$ by $\varsigma$ as well as those on $\varsigma$.  

\begin{prop}\label{prop:shield}
	There exists $c > 0$ such that the following holds for all $n\geq 1$. 
	For any $x \leq 0$ and $y \geq 0$ even, both sufficiently large,
	there exists a random variable $\calS = \calS(\omega,\omega')$ 
	that either takes value $\emptyset$ or is a shielding arc above $\La_n$ between $x$ and $y$ 
	and
	\begin{itemize}
		\item[(i)] if $\calS \neq \emptyset$, then 
		$\calS$ is measurable in terms of the edges in ${\rm Above}(\calS)$.
			In addition, $\calS = \emptyset$ is measurable in terms of the edges of $\bbH \setminus \La_n$;
		\item[(ii)] conditionally on $\omega'$, the event $\calS \neq \emptyset$ is decreasing in $\omega$;
		\item[(iii)] 
		if we call $\omega'$ {\em promising} when  $\bbP[\calS \neq \emptyset \,|\,\omega'] \geq c$, 
		then the event $\{\omega' \text{ is promising}\}$ is increasing in $\omega'$ and
		\[
			\bbP[\omega' \text{ is promising}] \geq c.
		\]
	\end{itemize}
\end{prop}

The last point implies that $\bbP[\calS\neq\emptyset] \geq c^2$. 
In addition to this bound, we would like to say that $\{\calS\neq\emptyset\}$ is increasing in $\omega'$ and decreasing in $\omega$, but this is unfortunately not the case, 
hence the more complicated expressions in (ii) and (iii).
We should mention that the ingredients already present in~\cite{GeoHig00} suffice to prove the existence of arbitrarily large shielding arcs. 
The crucial difference between the above and what is available with the arguments of~\cite{GeoHig00} is that shielding arcs exist with uniformly positive probability {\em between any $x$ and $y$}. This property, which relies on Proposition~\ref{prop:hp_perco}, will allow us to glue shielding arcs in $\bbH$ and $\bbH^-$ with no interference (see Corollary~\ref{cor:shield}). 

The rest of the section is dedicated to proving Proposition~\ref{prop:shield}. 
When there is no coexistence in $\bbH$ the result is relatively easy; we will mostly focus on the situation when there is coexistence in $\bbH$. We start with two preparatory results.

When there is coexistence in $\bbH$, 
the primal and the dual infinite clusters are separated by a unique semi-infinite interface starting at $\partial \bbH$.
Denote it by $\Gamma$ and parametrise starting from $\partial\bbH$.	
Formally $\Gamma$ is an edge-simple path on the medial lattice $\bbZ^2 + (1/2,1/2)$ of $\bbL$ (see Fig.~\ref{fig:arcs}), 
but the formal details in the definition of $\Gamma$ are of little importance and we omit them here.
Note that $\Gamma$ intersects $\partial \bbH$ only at its starting point $\Gamma_0$,
which is a point of the form $(n +1/2,0)$, with $n$ even when the infinite primal cluster is on the left and odd when it is on the right.

\begin{lem}\label{lem:hp_perco_pt}
	Suppose that $\phi$ is such that there is coexistence in $\bbH$, with the primal cluster to the left. 
	For $c>0$ from  Proposition~\ref{prop:hp_perco} and any even integer~$x$ and odd integer~$y$
	\begin{align}\label{eq:hp_perco_pt}
		\phi\big[(x,0)\xlra{\bbH} \infty \,\big|\, \Gamma\big] \geq c \qquad \text{ and }\qquad 
		\phi\big[(y,0) \xlra{\omega^*  \cap \bbH} \infty\,\big|\, \Gamma\big] \geq c,
	\end{align}
	whenever $\Gamma_0$ is to the right of $(x,0)$ and to the left of $(y,0)$, respectively.
\end{lem}

\begin{figure}
\begin{center}
\includegraphics[width = 0.6\textwidth]{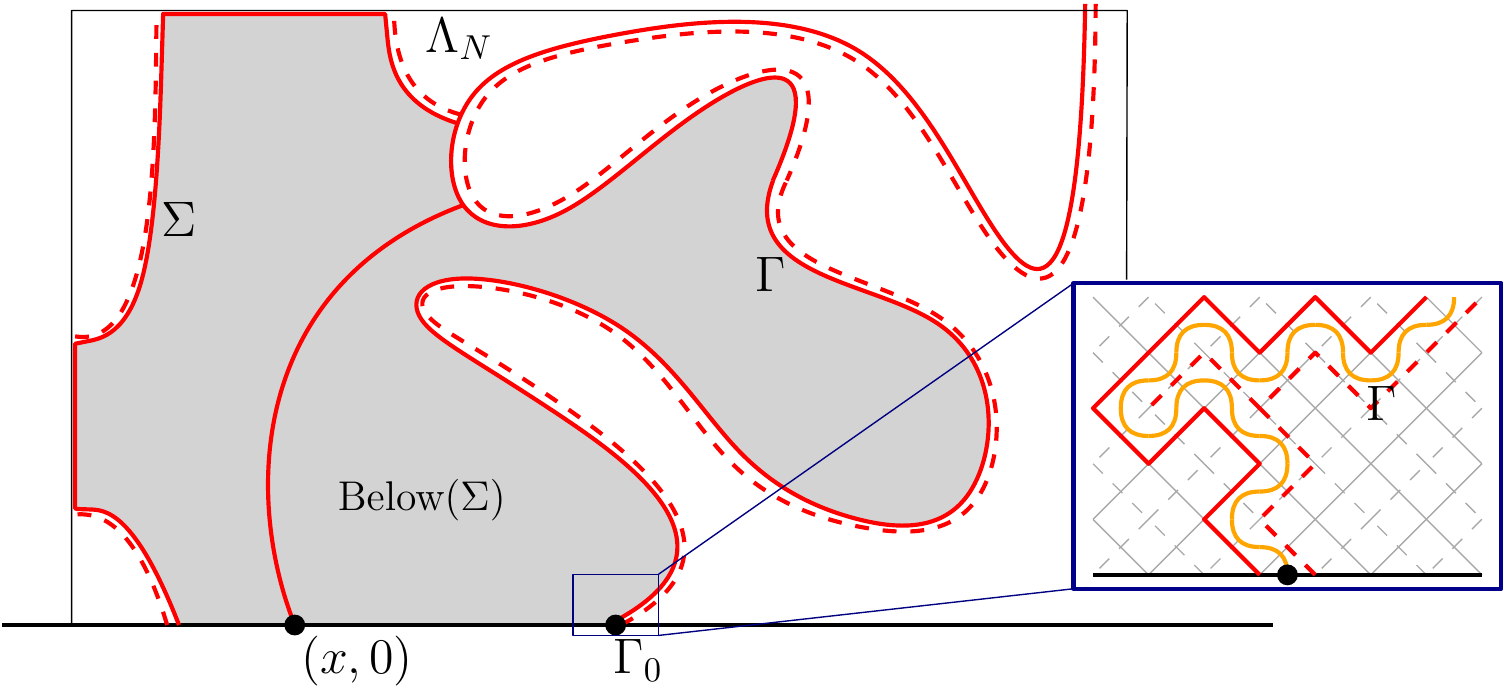}
\caption{The exploration path $\Gamma$ staring on the horizontal axis separates the infinite primal and dual clusters of $\bbH$. 
	For fixed $\La_N$, the outermost open path connecting $(-\infty,0] \times \{0\}$ to the left side of $\Gamma$ is denoted by $\Sigma$ and may be explored from the outside; the region ${\rm Below}(\Sigma)$ is marked in grey. Conditionally on $\Gamma$ and $\Sigma$, $(x,0)$ is connected to the wired arc above it with uniform positive probability.  }
	\label{fig:arcs}
\end{center}
\end{figure}

\begin{proof}
	We will focus on the first inequality and fix an even integer $x$. 
	Fix $\eps > 0$.
	
	Let $N = N (\Gamma) \geq 1$ be such that 
	\begin{align}\label{eq:yxz}
		\phi\big[(-\infty, x] \times \{0\} \xlra{\bbH \cap \La_N} \Gamma \,\big| \, \Gamma\big] \geq 1 - \eps,
	\end{align}
	whenever $\Gamma_0$ is to the right of $x$; otherwise set $N(\Gamma) = 0$. 
	
	Write $\Sigma$ for the exterior-most open primal path in $\bbH \cap \La_N$ that connects $(-\infty, x] \times \{0\}$ to $\Gamma$ if such a path exists;
	otherwise set $\Sigma = \emptyset$. 
	Denote by ${\rm Above}(\Sigma)$ the subgraph spanned by the edges of $\bbH$ separated from $(x,0)$ by $\Sigma \cup \Gamma$ as well as those on $\Sigma$ and those explored by $\Gamma$.
	Let ${\rm Below}(\Sigma)$ be the subgraph spanned by all the other edges. See Figure~\ref{fig:arcs} for an illustration. 
		
	Then, both $\Gamma$ and $\Sigma$  are measurable in terms of the edges in ${\rm Above}(\Sigma)$.
	It follows that, conditionally on $\Gamma$ and $\Sigma$, 
	the measure on ${\rm Below}(\Sigma)$ dominates the measure with wired boundary conditions on $\Sigma \cup \Gamma$ and free elsewhere.
	Using the above and Proposition~\ref{prop:hp_perco}, we find that 
	\begin{align*}
		\phi\big[(x,0)\xlra{\bbH} \Sigma \cup \Gamma \,\big|\, \Sigma, \Gamma \big]  \geq c,
	\end{align*}
	whenever $\Sigma \neq \emptyset$ and $\Gamma_0$ is to the right of $(x,0)$. 
	Notice that when the event above occurs, then $(x,0)$ is connected to infinity in $\bbH$, as it is connected to the primal side of the infinite path $\Gamma$, possibly via $\Sigma$. 
	
	Now, summing over all possible realisations of $\Sigma$ different from $\emptyset$  and using~\eqref{eq:yxz} we find 
	\begin{align*}
		\phi\big[(x,0)\xlra{\bbH} \infty \,\big|\, \Gamma \big]  \geq c (1- \eps).
	\end{align*}	
	Since $\eps > 0$ is arbitrary, we obtain the desired conclusion. 
\end{proof}

\begin{rem}\label{rem:hp_perco_primal}
	The same proof shows that if $\bbH$ contains no dual infinite cluster, 
	then for any even $x \in \bbZ$ 
	\begin{align}\label{eq:hp_perco_primal}
			\phi[(x,0)\xlra{\bbH} \infty] > c.
	\end{align}
	A similar statement holds for dual connections when $\bbH$ contains no primal infinite cluster.
\end{rem}

\begin{lem}\label{lem:hp_perco_pt2}
	Suppose that $\phi$ is such that there is coexistence in $\bbH$ with the primal infinite cluster to the left. 
	There exists $c' > 0$ such that, 
	for any $n \geq 1$ and sufficiently large $x \leq 0$ even and $y \geq 0$ odd,
	there exist $h \geq n$ and $m\in \bbZ$ such that  
	\begin{align}
		\phi\big[(x,0)\xlra{(\bbR\times [0,h]) \cap \La_n^c} [m, \infty) \times \{h\}\big] &> c' \qquad \text{and}\nonumber\\
		\phi\big[(y,0)\xlra{\omega^* \cap (\bbR\times [0,h]) \cap \La_n^c} (-\infty,m] \times \{h\}\big] &> c'. \label{eq:hp_perco_pt2}
	\end{align}
\end{lem}

\begin{proof}
	Let $\eps = c/12 > 0$, where $c$ is the constant from Proposition~\ref{prop:hp_perco} (and Lemma~\ref{lem:hp_perco_pt}).
	It suffices to prove the statement for $n$ large enough. 
	Fix $n \geq 1$ large enough that 
	\begin{equation}\label{eq:contour-through-lambda-n}
		\phi\big[\La_{n} \xlra{\bbH} \infty, \, \La_{n} \xlra{\omega^* \cap \bbH} \infty\big] > 1-\eps.
	\end{equation}
	
	To start, we claim that
	\begin{align}
    	\lim_{\substack{x \to - \infty\\ x \text{ even}}}\phi\big[(x,0)\xlra{\bbH} \infty \text{ but }(x,0)\nxlra{\bbH \cap \La_n^c } \infty\big] &= 0\quad  \text{ and} \nonumber\\
		\lim_{\substack{y \to \infty\\ y \text{ odd}}}\phi\big[(y,0)\xlra{\omega^*\cap\bbH} \infty \text{ but }(y,0)\nxlra{\omega^*\cap \bbH \cap \La_n^c } \infty\big] &= 0.\label{eq:no_avoiding_La}
    \end{align}
	The proof of these two equalities is identical, and we focus on the first one. 
	For $(x,0)$ to be connected to infinity in $\bbH$ but not in $\bbH \setminus \La_n$ 
	there needs to exist a dual path in $\bbH$ from $\La_n$ to $(-\infty, x] \times \{0\}$. 
	The existence of infinitely many such paths is incompatible with the existence of a primal cluster intersecting $(-\infty, x] \times \{0\}$ infinitely many times, and therefore has probability $0$. 
	Thus we deduce that only finitely many such dual paths exist a.s., and the Borel--Cantelli lemma implies~\eqref{eq:no_avoiding_La}. 
    	
	Fix now $x$ and $y$ as in the statement, large enough so that 
	\begin{equation*}
		\phi[(x,0)\xlra{\bbH} \infty \text{ but }(x,0)\nxlra{\bbH \cap \La_n^c } \infty] < \eps \text{ and }
		\phi[(y,0)\xlra{\omega^*\cap\bbH} \infty \text{ but }(y,0)\nxlra{\omega^*\cap \bbH \cap \La_n^c } \infty] < \eps.
	\end{equation*}
	
	For $h > n$ let $\calH_x(h)$ be the event that there exists $u \in \partial \La_n$
	such that both $u$ and $(x,0)$ are connected to  $\bbR \times \{h\}$ in $(\bbR \times [0,h]) \setminus \La_n$, but not to each other. 
	Define $\calH_y(h)$ similarly for $y$ and involving dual connections. 
	Due to the uniqueness of the primal and dual infinite clusters in $\bbH$, we may pick $h = h(n,x,y)$ large enough such that
	\begin{align*}
		\phi[\calH_x(h)] \leq \eps \qquad \text{ and }\qquad \phi[\calH_y(h)] \leq \eps.
	\end{align*}

	Recall the definition of the infinite interface $\Gamma$ separating the primal and dual infinite clusters in $\bbH$.
	Define $T_{h}$ as the first time when $\Gamma$ intersects $\bbR\times \{h\}$ and
	let $M_{h}$ be the even integer closest to the first coordinate of $\Gamma_{T_{h}}$.
	Define $m = m(n,h,x,y) \in \bbZ$ as the median value for $M_{h}$, so that
	\begin{equation}\label{eq:median-point-at-height-h}
		\phi(M_{h} \geq m) \geq \tfrac{1}{2} \quad \text{and} \quad \phi(M_{h} \leq m) \geq \tfrac{1}{2}.
	\end{equation}
	
	Now that the quantities $h$ and $m$ are fixed, let us prove~\eqref{eq:hp_perco_pt2}.
	We will focus on the bound for $x$. 
	By Lemma~\ref{lem:hp_perco_pt}, 
	\begin{equation*}
		\phi\big[(x,0)\xlra{\bbH} \infty \,\big|\, \Gamma_0 \in \La_n \text{ and }M_{h} \geq m \big] \geq c.
	\end{equation*}
	Now, due to~\eqref{eq:contour-through-lambda-n} and the choice $m$, the event in the conditioning has probability at least $1/2 - \eps$, 
	and we find 
	\begin{equation*}
		\phi\big[(x,0)\xlra{\bbH} \infty ,\, \Gamma_0 \in \La_n \text{ and }M_{h} \geq m \big] \geq c(1/2 - \eps).
	\end{equation*}	
	Finally, we have 
	\begin{align*}
		&\phi\big[(x,0)\xlra{(\bbR\times [0,h]) \cap \La_n^c} [m, \infty) \times \{h\}\big] \\
    	&\quad\geq \phi\big[(x,0)\xlra{(\bbR\times [0,h]) \cap \La_n^c } \Gamma \text{ and }M_{h} \geq m \big] \\
    	&\quad\geq \phi\big[(x,0)\xlra{\bbH} \infty, \,  \Gamma_0 \in \La_n \text{ and }M_{h} \geq m \big] 
    	-\phi\big[(x,0)\xlra{\bbH} \infty \text{ but } (x,0)\nxlra{\bbH \cap \La_n^c} \infty] 
		-\phi\big[\calH_x(h)] \\
		&\quad\geq  c/2 - 3\eps.
	\end{align*}
	Due to our choice of $\eps$, we conclude that the above has at least probability $c' = c/4 > 0$. 
	
	The same argument applies to $y$ instead of $x$, {\em with the same value of $m$}. 
	Notice that when $M_h \leq m$, the first dual vertex on $\bbR \times \{h\}$ visited by $\Gamma$ may be $(m+1,h)$.
	In order to yield a $\omega^{*}$-connection from $(y,0)$ to $(-\infty,m]\times \{h\}$, we may modify the configuration in the neighbourhood of $(m,h)$ by the finite energy property.
	\end{proof} 
	
\begin{proof}[Proposition~\ref{prop:shield}]
	Fix $n$. For simplicity, let us assume that $\phi'$ is the translate of $\phi$ by $(2,0)$; other translates may be treated in a similar way. 
	We split the proof into several cases. 
	\medskip 
	
	\noindent (a) If $\bbH$ contains no dual infinite cluster $\phi$-a.s., then the same is valid for $\phi'$. 
	Due to Remark~\ref{rem:hp_perco_primal},
	\begin{align*}
		\liminf_{x \to -\infty }\phi'[(x,0)\xlra{\bbH \setminus \La_n} \infty] \geq c \quad \text{ and }\quad 
		\liminf_{y \to \infty } \phi'[(y,0)\xlra{\bbH \setminus \La_n} \infty] \geq c.
	\end{align*} 
	Indeed, the absence of an infinite dual cluster in $\bbH$ ensures that
	connections to infinity for far enough points $(x,0),(y,0) \in 2\bbZ \times \{0\}$ may be chosen to avoid $\La_n$ 
	(see also the proof of~\eqref{eq:no_avoiding_La} for a similar argument). 
	
	Moreover, due again to the absence of a dual infinite cluster in $\bbH$ and to~\eqref{eq:FKG},
	\begin{align*}
		\phi'[(x,0)\xlra{\bbH \setminus \La_n}(y,0)] 
		\geq \phi'[(x,0)\xlra{\bbH \setminus \La_n}\infty]\, \phi'[(y,0)\xlra{\bbH \setminus \La_n}\infty] 
		\geq c^2/2,
	\end{align*}
	for all $x < 0$ and $y > 0$ even and large enough. 
	
	Fix now $R = R(n,x,y)$ such that the above connection between  $(x,0)$ and $(y,0)$  occurs in $\La_R$ with probability at least $c^2/4$. 
	Then define $\calS$ as the exterior-most path contained in $\omega'$, 
	with endpoints $(x,0)$ and $(y,0)$ and which is contained in $\bbH \cap \La_R \setminus \La_n$. 
	If no such connection exists, set $\calS = \emptyset$. 
	
	It is immediate that $\calS$ is a shielding arc and that it is measurable in terms of the edges in ${\rm Above}(\calS)$.
	Moreover, the existence of $\calS$ is increasing in $\omega'$, is independent of $\omega$ 
	and occurs with probability at least $c^2/4$. 
	\medskip
	
	\noindent 
	(b) If $\bbH$ contains no primal infinite cluster $\phi$-a.s., apply the same construction as above to $\omega^*$ instead of $\omega'$. 
	\medskip
	
	\noindent (c) If there is coexistence in $\bbH$ with the primal cluster to the left, proceed as follows. 
	Let $x < 0$ and $y+1 > 0$ be even and odd, respectively, and large enough that Lemma~\ref{lem:hp_perco_pt2} applies to them;
	let $h,m$ be the integers given by  Lemma~\ref{lem:hp_perco_pt2}.

	\begin{figure}
	\begin{center}
	\includegraphics[width = 0.6\textwidth]{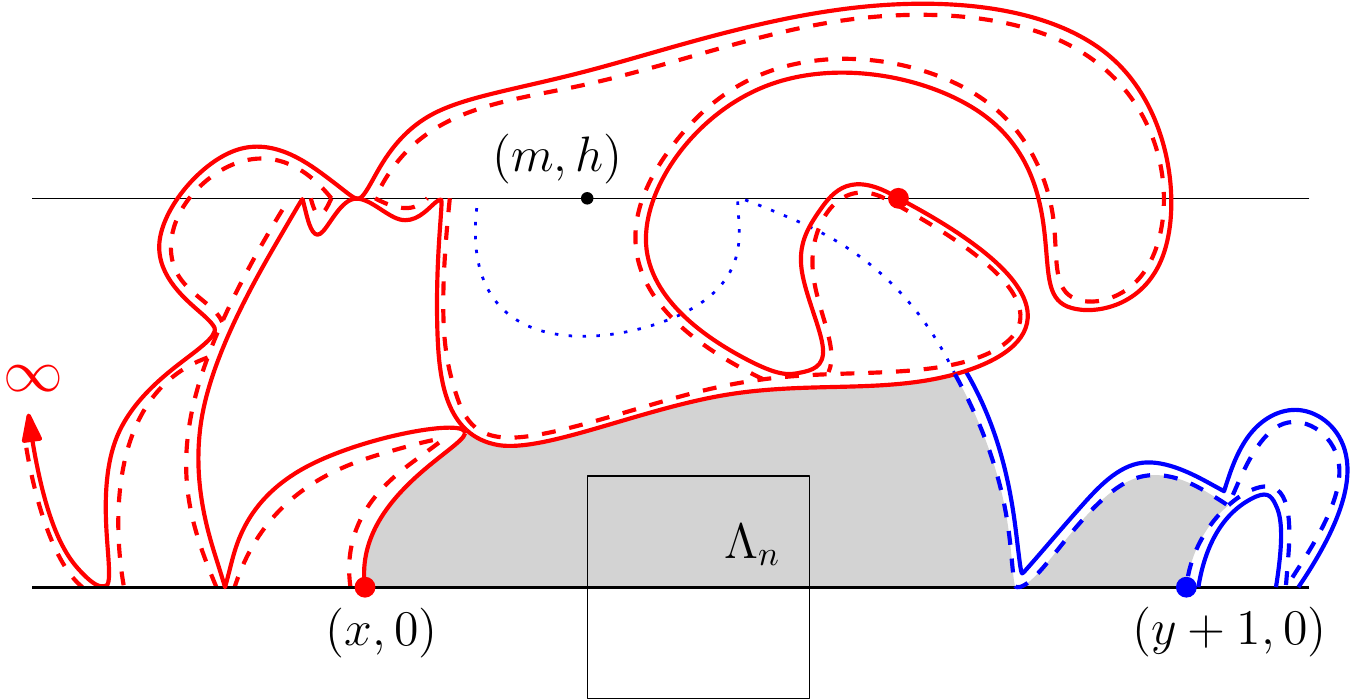}
	\caption{The blue lines refer to $\omega$ and the red ones to $\omega'$; the solid lines are open primal paths, while the dashed ones are dual. 
	In the first step of the exploration, the red interface is discovered, starting at $(x,0)$ and up to the first time it hits $[m,\infty) \times \{h\}$.
	The rest of the red interface is explored in the second step and the blue interface in the third step.
	For the last step to be successful it suffices that $(y+1,0)$ is connected to $(-\infty,m] \times \{h\}$ in $\omega^* \cap (\bbR \times[0,h]) \cap  \La_n^c$.
	The path $\calS$ is then the upper boundary of the grey region. }
	\label{fig:shield}
	\end{center}
	\end{figure}
	
	We will now describe an exploration algorithm that, if successful, will create $\calS$. If the algorithm fails, we set $\calS = \emptyset$. 
	See Figure~\ref{fig:shield} for an illustration.
	\smallskip 
	
	\noindent {\bf Step 1:}
	Explore the left-most connection in $\omega'$ between $(x,0)$ and $[m,+\infty) \times \{h\}$ which is contained in $(\bbR \times[0,h]) \setminus \La_n$. 
	If no such connection is found, the algorithm fails.
	\smallskip 
	
	\noindent {\bf Step 2:}
	From the tip of the previous exploration, continue exploring the left-most infinite path in $\omega' \cap (\bbH \setminus \La_n)$. 
	If no such path is found, the algorithm fails.
	\smallskip 
	
	\noindent {\bf Step 3:}
	Explore the right-most path in $\omega^*$ starting at $(y+1,0)$, ending at a previously explored edge and contained in  $(\bbR \times[0,h]) \setminus \La_n$. 
	If no such connection is found, the algorithm fails.
	\medskip 
	
	When the algorithm is successful, let $D$ be the connected component of $[x,y+1] \times \{0\}$ in the set of unexplored edges of $\bbH$
	and let $\calS$ be the curve forming the boundary of $D$ other than  $[x,y+1] \times \{0\}$.
	Then $\calS$ is a shielding arc. 
	Indeed, after Step 2, the connected component of $(y+1,0)$ in the unexplored edges of $\bbH$ is bounded entirely by explored edges contained in $\omega'$. 
	Thus, the $\omega^*$-open path explored in Step 3 forms $\calS^{(2)}$, while $\calS^{(1)} := \calS \setminus \calS^{(2)}$ is formed of edges of $\omega'$.
	Moreover, the construction of $\calS$ shows that $\calS$ is measurable in terms of the edges in ${\rm Above}(\calS)$.
	
	Now let us argue that points $(ii)$ and $(iii)$ are satisfied. 
	Note that $\omega'$ is promising if and only if both steps 1 and 2 are successful (otherwise $\calS = \emptyset$ for any $\omega^*$). 
	That Step 1 and Step 2 are successful are both increasing events depending only on $\omega'$. 
	Lemmas~\ref{lem:hp_perco_pt} and~\ref{lem:hp_perco_pt2} provide uniformly positive lower bounds for the probabilities for each of these events. 
	Applying positive association, we conclude that $\omega'$ is promising with uniformly positive probability.
	Finally, as an intersection of increasing events, $\{\omega' \text{ promising}\}$ is also increasing. 

	Fix a promising $\omega'$.
	It is immediate that $\{\calS \neq \emptyset\}$ is decreasing in $\omega$ (for fixed $\omega'$).
	Moreover, for Step 3 to be successful it suffices (but is not generally necessary) that $(y+1,0)$ 
	be connected  to $(-\infty,m] \times \{h\}$ by a $\omega^*$-open path contained in $(\bbR \times[0,h]) \setminus  \La_n$.
	This event occurs with uniformly positive probability due to Lemma~\ref{lem:hp_perco_pt2}.
	Thus, $\calS \neq \emptyset$ has uniformly positive probability conditionally on any promising $\omega'$.
	\medskip
	
	\noindent (d) Apply the horizontal reflection of the construction of point (c). 
\end{proof}

\subsection{Proof of Proposition~\ref{prop:Gibbs=trans_inv}}

To start, let us state a simple consequence of Proposition~\ref{prop:shield}.
Notice that the equivalent of Proposition~\ref{prop:shield} also applies in $\bbH^-$.

\begin{cor}\label{cor:shield}
	There exists $c > 0$ such that the following holds for all $n\geq 1$. 
	For any $x \leq 0$ and $y \geq 0$, both sufficiently large,
	if we write $\calS$ and $\calS^-$ for the shielding arcs provided by Proposition~\ref{prop:shield} applied in $\bbH$ and $\bbH^-$,
	respectively, then
    \begin{align*}
    	\bbP[\calS \neq \emptyset,\, \calS^- \neq \emptyset ] \geq c.
    \end{align*}
\end{cor}

\begin{proof}
	Let $c>0$ be the constant given by Proposition~\ref{prop:shield}. 
	Fix $n$ and $x < 0$ and $y > 0$ even and large enough for Proposition~\ref{prop:shield} to apply. 
		
	Since the events $\{\omega' \text{ is promising for $\bbH$}\}$ and $\{\omega' \text{ is promising for $\bbH^-$}\}$ 
	are both increasing in $\omega'$, we conclude by positive association that
	\begin{align*}
		\bbP[\omega' \text{ promising for both $\bbH$ and $\bbH^-$}] \geq c^2. 
	\end{align*}
	Now, for any such $\omega'$, using the positive association of $\phi$ and property (ii) of  Proposition~\ref{prop:shield},
	\begin{align*}
		\bbP[\calS \neq \emptyset \text{ and }\calS^- \neq \emptyset \,|\, \omega'] \geq 
		\bbP[\calS \neq \emptyset \,|\, \omega']\,\bbP[\calS^- \neq \emptyset \,|\, \omega'] \geq c^2. 
	\end{align*}
	The conclusion follows by combining the two inequalities above and by adjusting $c$. 	
\end{proof}

\begin{proof}[Proposition~\ref{prop:Gibbs=trans_inv}]
	We will show that 
	\begin{align*}
		\phi'[A] \geq \phi[A]
	\end{align*}
	for any increasing event $A$ that depends on a finite set of edges.
	Then the opposite inequality holds by symmetry, and we deduce that $\phi$ and $\phi'$ are equal for all increasing cylinder events. 
	The monotone class theorem allows then to deduce that $\phi = \phi'$. 
	For the rest of the proof we focus on showing the inequality above 
	
	Fix such an event $A$ and let $n_{0}$ be such that $A$ depends only on the edges in $\La_{n_{0}}$. 
	Define an increasing sequence of scales $(n_k)_{k \geq 0}$ as follows. 
	Assuming $n_0,\dots n_{k-1}$ are defined, 
	let $x_k<0<y_k$ be sufficiently large for Corollary~\ref{cor:shield} to apply with $n = n_{k-1}$. 
	Let $\calS_{k}$ and $\calS_{k}^{-}$ be the shielding arcs produced by Corollary~\ref{cor:shield} with $x=x_k$, $y=y_k$, $n= n_{k-1}$.
    Then, we may choose $n_{k} > n_{k-1}$ such that
    \begin{align}\label{eq:GG'}
    	\bbP[\calS_{k} \neq \emptyset,\, \calS_{k}^- \neq \emptyset \text{ and } \calS_{k},\calS_{k}^- \subset \La_{n_k} ] \geq c/2,
	\end{align}
	where $c > 0$ is the constant of Corollary~\ref{cor:shield}.
	We will say that $n_k$ is a {\em successful scale} if the event above holds. 
	
	As a product of two tail trivial measures, $\bbP$ is also tail trivial by Fubini's theorem. 
	Combined with~\eqref{eq:GG'}, the tail triviality of $\bbP$ implies that there exists a.s. infinitely many successful scales. 
	
	For $\eps > 0$, let $K = K(n_0,\eps)$ be such that
	\begin{align*}
		\bbP[\exists j \in [1, K] \text{ such that $n_j$ is successful}] \ge 1-\eps. 
	\end{align*}
	
	Let $J \in [1,K]$ be maximal such that $n_J$ is successful. 
	Write $\calS$ and $\calS^-$ for the paths generated by Corollary~\ref{cor:shield} applied to $x_J,y_J$ and $n = n_{J-1}$. 
	Then $\calS \cup \calS^-$ delimits a finite domain $\calD \subset \La_{n_J}$ which contains $\La_{n_0}$. 
	When no scale $n_j$ with $1 \leq j\leq K$ is successful, set $\calS = \calS^{-} =  \calD = \emptyset$.
	 
	We now show that $\calS$ and $\calS^-$ are measurable in terms of the edges outside of $\calD$.
	Indeed, for any fixed $j$, $J = j$ is the intersection of the event that $n_j$ is successful and of 
	the event that no scale $n_k$ with $j < k \leq K$ is successful. 
	The latter depends only on the edges outside of $\La_{n_j}$, and therefore outside of $\calD$. 
	The former depends on the shielding arcs $\calS_{j}$ and $\calS_{j}^-$, which were shown in Proposition~\ref{prop:shield} to depend only on the edges outside of $\calD$.  
	
	Let $\xi$ and $\xi'$ be the boundary conditions induced on $\partial \calD$ by $\omega \cap \calD^c$ and $\omega'\cap \calD^c$, respectively. 
	Next we claim that $\xi \leq \xi'$. 
	
	The proof of this fact depends on the situations in $\bbH$ and $\bbH^-$. 
	It is always the case that $\partial \calD$ may be split into at most four arcs $\calS^{(1)},\dots, \calS^{(4)}$
	of alternating type: $\calS^{(1)}$ and $\calS^{(3)}$ are formed of primal edges which are in $\omega'$, 
	while  $\calS^{(2)}$ and $\calS^{(4)}$ are formed of dual edges which are in $\omega^*$; 
	these arcs are joined by single edges of $\bbZ^2$. 
	When the boundary of $\calD$ is formed of only one or two arcs, we set $\calS^{(3)} = \calS^{(4)} = \emptyset$ 
	and potentially $\calS^{(1)} = \emptyset$ or $\calS^{(2)} = \emptyset$.
	
	If both $\bbH$ and $\bbH^-$ contain no infinite primal cluster, then $\partial \calD$  is a circuit in $\omega^*$,
	and $\xi$ is the free boundary condition. Since the free boundary condition is minimal, $\xi \leq \xi'$.
	Similarly, when both $\bbH$ and $\bbH^-$ contain no infinite dual cluster, then $\partial \calD$ is a circuit of primal edges of $\omega'$, and again  $\xi \leq \xi'$.
	The former case is realised when $\phi = \phi^{0}$ and the latter one when $\phi = \phi^{1}$.
	In retrospect, all other cases are excluded. 
	
	The boundary of $\cal\calD$ is formed of two arcs in most other cases:
	when $\bbH$ and $\bbH^-$ each contain a single infinite cluster but of different types, 
	when one of $\bbH$ and $\bbH^-$ contains a single infinite cluster and the other has coexistence,
	or when the both $\bbH$ and $\bbH^-$ have coexistence, with the primal cluster on the same side. 
	In all of these cases $\xi'$ is wired on $\calS^{(1)}$ and is arbitrary on the rest of $\partial \calD$, 
	while $\xi$ is free on $\calS^{(2)}$ and arbitrary on $\calS^{(1)}$. It follows directly that $\xi' \geq \xi$. 
	
	Finally, for $\partial \calD$ to have four arcs, both  $\bbH$ and $\bbH^-$ need to have coexistence, 
	with the primal cluster on opposite  sides. 
	Notice then that our construction of shielding arcs guarantees that $\calS^{(1)}$ and $\calS^{(3)}$
	are both connected to infinity in $\omega' \setminus \calD$. 
	Due to our convention for accounting for infinite clusters in~\eqref{eq:DLR-FK}, the arcs $\calS^{(1)}$ and $\calS^{(3)}$ are wired together in $\xi'$. 
	As such, $\xi'$ dominates any boundary condition that is free on $\calS^{(2)}$ and $\calS^{(4)}$. 
	Thus we conclude that $\xi \leq \xi'$ also in this case.
	 
	Now, due to~\eqref{eq:DLR-FK} applied to $\phi$ and $\phi'$ and to the fact that $\calD$ is explored from the outside, 
	we find
	\begin{align*}
		\phi'[A]-\phi[A]
		\geq \bbE\big[(\phi_{\calD}^{\xi'}[A] - \phi_{\calD}^{\xi}[A])\1_{\{\calD \neq \emptyset\}} \big] -\bbP[\calD =  \emptyset]
		\geq -\eps.
	\end{align*}
	Since $\eps$ is arbitrary, we conclude that $\phi'[A] \geq\phi[A]$, which was our goal. 
\end{proof}

\section{Thermodynamical limits for FK-percolation: proof of Corollary~\ref{cor:every-limit-is-gibbs}}
\label{sec:every-limit-is-gibbs}

In this section we prove Corollary~\ref{cor:every-limit-is-gibbs}. 
The proof is immediate when $\phi_{p,q}^{0} = \phi_{p,q}^{1}$, since for any finite subgraph $D$ of $\bbL$ and any boundary conditions $\xi$
\begin{align*}
	\phi^0_{D,p,q} \leq_{\rm st}\phi^\xi_{D,p,q} \leq_{\rm st}\phi^1_{D,p,q},
\end{align*}
and the free and wired measures on $D$ have the same limit as $D$ increases to $\bbL$.
Henceforth we focus on $p\in (0,1)$ and $q\geq 1$ for which $\phi_{p,q}^0\neq \phi_{p,q}^1$; the proof uses the following lemmas.

For $\xi$ a boundary condition on some finite subgraph $D = (V,E)$ of $\bbL$, 
a wired component of $\xi$ is an element of the partition $\xi$ that contains at least two vertices. 
We start by proving Corollary~\ref{cor:every-limit-is-gibbs} for sequences of boundary conditions with at most one wired component. 

\begin{lem}\label{lem:every-limit-is-gibbs1}
	Let $p\in (0,1)$ and $q\geq 1$ be such that $\phi_{p,q}^0\neq \phi_{p,q}^1$.
	Let $(D_{n})_{n\geq 1}$ be a sequence of domains increasing to $\bbL$ 
	and $(\xi_{n})_{n\geq 1}$ be a sequence of boundary conditions containing at most one wired component.
	Assume that the sequence of FK-percolation measures $\phi_{D_{k},p,q}^{\xi_{k}}$ has a weak limit $\phi$.
	Then $\phi$ is a Gibbs measure for FK-percolation with parameters $p,q$, and 
	\begin{align}\label{eq:noinf_primal_dual}
		\phi[0\xlra{} \infty \text{ and }(1,0)\xlra{*} \infty] = 0.
	\end{align}
\end{lem}

\begin{proof}
	Fix $p\in (0,1)$ and $q\geq 1$ for which $\phi_{p,q}^0\neq \phi_{p,q}^1$ and omit them from notation for the rest of the proof. 
	Let $(D_{n}, \xi_{n})_{n\geq 1}$ and $\phi$ be as in the statement. 
	We will prove that $\phi$ is a Gibbs measure for FK-percolation; 
	\eqref{eq:noinf_primal_dual} then follows from Theorem~\ref{thm:mainFK}.
	Fix a finite subgraph $G$ of $\bbL$. 
	
	To start, we show that there exists $C_0 = C_0(G) > 0$ such that for any $R$ and $n$ such that  $G \subset \La_R \subset D_n$, 
	\begin{align}\label{eq:bcbc}
		\phi_{D_n}^{\xi_n}\big[	\exists x,y \in \partial G \text{ s.t. } 
		x\nxlra{\omega^{\xi_n} \setminus G} y \text{ but } x\xlra{ G^c} \partial \La_R \text{ and } y\xlra{ G^c} \partial \La_R \big]
		\leq C_0 e^{-c R},
	\end{align}
	where $c >0$ is given by~\eqref{eq:exp_decay}.
	
	\begin{figure}
	\begin{center}
	\includegraphics[width = 0.6\textwidth]{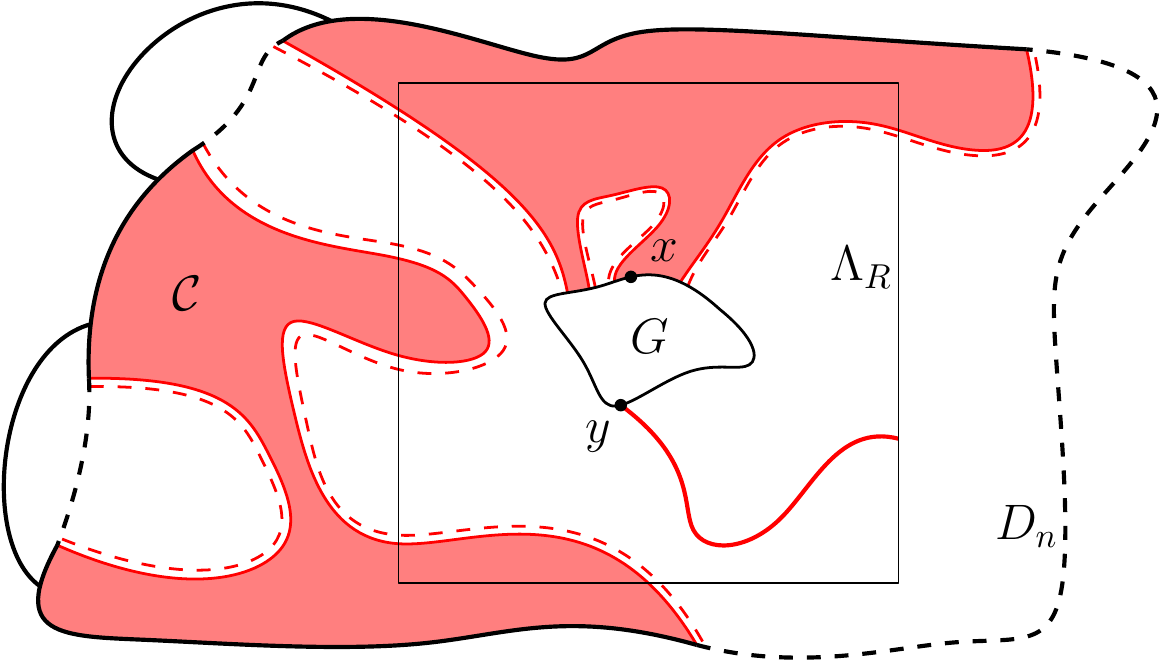}
	\caption{A domain $D_n$ with boundary conditions containing a single wired component; the wired component is formed of three arcs which are wired together. 
	If $x$ and $y$ are both connected to $\partial \La_R$, but not to each other, then at least one of them is connected to $\partial \La_R$,
	but not to the wired component of $\partial D_n$; in the figure, $y$ is such.  
	After exploring the cluster of the wired component of $\partial D_n$ (in red), $y$ is (almost) surrounded by free boundary conditions.  }
	\label{fig:one_wired_comp}
	\end{center}
	\end{figure}
	
	Fix $R$ and $n$ as above and write $\calC$ for the connected component of the wired component of $\xi_n$ in $\omega \setminus G$.  
	Then~\eqref{eq:SMP} implies that, for any realisation  ${\rm C}$ of $\calC$, 
	the restriction of $\phi_{D_n}^{\xi_n}[.|\calC = {\rm C}]$ to $D_n \setminus {\rm C}$ is $\phi_{D_n \setminus {\rm C}}^{\zeta}$, 
	where $\zeta$ is free at all vertices of $\partial (D_n \setminus {\rm C})$ not contained in $G$ and wired on $\partial (D_n \setminus {\rm C})\cap G$. See Figure~\ref{fig:one_wired_comp} for an illustration. 
	
	Now, for the event in~\eqref{eq:bcbc} to occur, it is necessary that at least one vertex of $\partial G$ 
	is connected to $\partial \La_{R}$ in  $D_n \setminus {\rm C}$. 
	Indeed, for $x$ and $y$ to be disconnected in $\omega^{\xi_n} \setminus G$, at least one of them needs to not be contained in ${\rm C}$. 
	Thus, 
	\begin{align*}
		\phi_{D_n}^{\xi_n}\big[	\exists x,y \in \partial G \text{ s.t. } &
		x\nxlra{\omega^{\xi_n} \setminus G} y \text{ but } x\xlra{ G^c} \partial \La_R \text{ and } y\xlra{ G^c} \partial \La_R \big]\\
		&\leq \sum_{{\rm C}} \phi_{D_n \setminus {\rm C}}^{\zeta}[G\lra \partial \La_{R}] \,\phi_{D_n}^{\xi_n}[\calC = {\rm C}]\\
		&\leq \sum_{{\rm C}} { C}_1 \phi_{D_n \setminus {\rm C}}^{0}[G\lra \partial \La_{R}] \,\phi_{D_n}^{\xi_n}[\calC = {\rm C}]\\
		&\leq \sum_{{\rm C}} { C}_1 \phi^{0}[G\lra \partial \La_{R}] \,\phi_{D_n}^{\xi_n}[\calC = {\rm C}]\\
		&\leq { C}_1 \, |\partial G| \, \exp\big[-c (R -{\rm rad}(G))\big].
	\end{align*}	
	where the sums are over all possible realisations ${\rm C}$ of $\calC$.
	The second inequality follows by modifying the boundary conditions on $\partial (D_n \setminus {\rm C}) \cap G$, 
	which affect the probability of the event by a constant $C_1$ depending only on $|\partial G|$.   
	The third inequality uses~\eqref{eq:CBC}.
	The last inequality is simply~\eqref{eq:exp_decay}; 
	${\rm rad}(G)$ stands for the minimal $r$ such that $G \subset \La_r$. 
	This concludes the proof of~\eqref{eq:bcbc}. \medskip
	
	We turn to the proof of~\eqref{eq:DLR-FK} for $\phi$.
	Fix a boundary condition $\chi$ on $G$ and an event $B \in \calF(G^c)$.
	Assume $\phi[\{\omega \text{ induces $\chi$ on $G$}\}\cap B] >0$. 
	Our goal is to prove that 
	\begin{align}\label{eq:zsup0}
		    \phi[A \cap \{\omega \text{ induces $\chi$ on $\partial G$}\} \cap B] =
		    \phi_{G}^\chi[A] \phi[\{\omega \text{ induces $\chi$ on $\partial G$}\} \cap B],
	\end{align}
	for any event $A \in \calF(G)$. 
	By the monotone class theorem, it suffices to prove the above for events $B$ that depend on finitely many edges. 
	We assume this henceforth.
		
	Fix some $\eps > 0$ and let $R \geq 1$ be such that
	$G \subset \La_R$, $B \in \calF(\La_R\setminus G)$ and for any $n$ such that $\La_{R} \subset D_n$, 
	\begin{align}\label{eq:zsup}
		\phi_{D_n}^{\xi_n}\big[	\exists x,y \in \partial G \text{ s.t. } &
		x\nxlra{\omega^{\xi_n} \setminus G} y \text{ but } x\xlra{ G^c} \partial \La_R \text{ and } y\xlra{ G^c} \partial \La_R \big] < \eps.
    \end{align} 
    The existence of such an $R$ follows from our choice of $B$ and~\eqref{eq:bcbc}.
    Notice that, due to the convergence, the above also holds for $\phi$. 
       
	Fix now $n$ large enough that $\La_{R} \subset D_n$ and 
	\begin{align}\label{eq:zsup2}
		\sup_{C \in \calF(\La_{R})}\big|\phi_{D_n}^{\xi_n}[C] - \phi[C]\big| < \eps.
	\end{align}
	The existence of such an $n$ is a consequence of the notion of weak convergence for the product topology. 	
	
	For $\omega$ a configuration on $\bbL$ or $D_n$,	
	write $\zeta(\omega)$ for the boundary conditions on $\partial G$ induced by $\omega$
	and $\zeta_R(\omega)$ for those induced by $\omega \cap \La_R$, where the boundary of $\La_R$ is considered wired. 
	That is, $x,y \in \partial G$ are wired in $\zeta_R(\omega)$ if they are connected to each other in $\omega \cap \La_R \cap G^c$, 
	or if they are both connected to $\partial \La_R$ in $\omega \cap G^c$. 
	Then, for any event~$A \in \calF(G)$,
	\begin{align}
        \phi[A \cap \{\zeta(\omega) = \chi\} \cap B] 
        &= \phi[A \cap \{\zeta_R(\omega) = \chi\} \cap B]  + \delta_1, \nonumber\\
        &= \phi_{D_n}^{\xi_n}[A \cap \{\zeta_R(\omega) = \chi\} \cap B]  + \delta_1 + \delta_2 \nonumber\\
        &= \phi_{D_n}^{\xi_n}[A \cap \{\zeta(\omega^{\xi_n})= \chi \} \cap B]  + \delta_1 + \delta_2 +\delta_3 \nonumber\\
        &= \phi_{G}^{\chi}[A]\, \phi_{D_n}^{\xi_n}[\{\zeta(\omega^{\xi_n})= \chi\} \cap B]  + \delta_1 + \delta_2 +\delta_3.
        \label{eq:zsup5}
	\end{align}
	where $\delta_1$, $\delta_2$ and $\delta_3$ are quantities defined by the successive differences of the probabilities above
	and the last equality is the application of~\eqref{eq:SMP} for $\phi_{D_n}^{\xi_n}$.
	We bound $\delta_1$, $\delta_2$ and $\delta_3$ as follows. 
	For $\zeta_R(\omega)$ and $\zeta(\omega^{\xi_n})$ to differ, the event in~\eqref{eq:zsup} needs to occur.
	Thus $|\delta_3|\leq \eps$. Since any two infinite clusters are considered to be wired together, the same argument can be applied to $\phi$ to show that $|\delta_1 | \leq \eps$.
	Lastly, by~\eqref{eq:zsup2}, $|\delta_2 | \leq \eps$.

	Applying ~\eqref{eq:zsup5} to the full event, we find 
	\begin{align*}
       \big| \phi[\{\zeta(\omega) = \chi\} \cap B] - 
       \phi_{D_n}^{\xi_n}[\{\zeta(\omega^{\xi_n})= \chi\} \cap B] \big| \leq 3\eps.
	\end{align*}
	Finally, combining the above with~\eqref{eq:zsup5}, yields
	\begin{align*}
        \big| \phi[A \cap \{\zeta(\omega) = \chi\} \cap B] 
        - \phi_{G}^{\chi}[A]\, \phi[\{\zeta(\omega) = \chi\} \cap B]\big| \leq 6\eps.
	\end{align*}
	Since $\eps$ is arbitrary, the above implies~\eqref{eq:zsup0}, and the proof is complete. 
\end{proof}

\begin{lem}\label{lem:every-limit-is-gibbs2}
		Fix $p\in (0,1)$ and $q\geq 1$ for which $\phi_{p,q}^0\neq \phi_{p,q}^1$.
		For any $K \geq 1$ and $\eps >0$, there exists $N = N(K,\eps)\geq 1$ such that,
		for any domain $D$ containing $\La_N$ and any boundary conditions $\xi$ on $D$ with at most $K$ wired components,  
    	\[
    		\phi_{D,p,q}^\xi[0\xlra{} \partial D \text{ and }(1,0)\xlra{*} \partial D ] < \eps.
    	\]
\end{lem}

\begin{proof}
	Fix $p$ and $q$ and omit them from the notation. 
	Let us first treat the case $K =1$. We proceed by contradiction, that is, we suppose the existence of $\eps > 0$
	and of a sequence of finite graphs $D_n$ increasing to $\bbL$ and boundary conditions $\xi_n$, each with at most one wired component,
	and such that 
	\begin{align}\label{eq:absasumpt}
		\phi_{D_n}^{\xi_n}[0\xlra{} \partial D_n \text{ and }(1,0)\xlra{*} \partial D_n ] \geq \eps \qquad \text{ for all $n\geq 1$}.
	\end{align}
	The space $\{0,1\}^{E(\bbL)}$ is compact in the product topology, 
	and therefore the sequence of measures $\phi_{D_n}^{\xi_n}$ has subsequential weak limits. 
	If $\phi$ is such a limit, Lemma~\ref{lem:every-limit-is-gibbs1} shows that 
	$\phi[0\xlra{} \infty \text{ and }(1,0)\xlra{*} \infty ] =0$, which contradicts~\eqref{eq:absasumpt}. 
	This concludes the proof for~$K = 1$. 
	\medskip 
	
	Let now $K$ be arbitrary. For a finite graph $D$ and boundary conditions $\xi$, 
	write $\tilde\xi$ for the boundary conditions obtained by wiring together all wired components of $\xi$. 
	If $\xi$ has at most $K$ wired components, then
	\begin{align*}
		\phi_{D}^{\xi}[0\xlra{} \partial D \text{ and }(1,0)\xlra{*} \partial D ] 
		\leq q^K \phi_{D}^{\tilde\xi}[0\xlra{} \partial D \text{ and }(1,0)\xlra{*} \partial D ].
	\end{align*}
	Note that $\tilde \xi$ has at most one wired component, and we may apply the case $K=1$ to the probability on the right-hand side. 
	In conclusion, one may choose $N$ large enough to render the right-hand side of the above arbitrarily small, 
	which concludes the proof of the lemma. 
\end{proof}

\begin{lem}\label{lem:every-limit-is-gibbs3}
		Fix $p\in (0,1)$ and $q\geq 1$ for which $\phi_{p,q}^0\neq \phi_{p,q}^1$.
		For any $\eps >0$, there exists $N = N(\eps)\geq 1$ such that,
		for any domain $D$ containing $\La_N$ and any boundary conditions $\xi$ on $D$,
    	\begin{align*}
    		\phi_{D,p,q}^\xi[0\xlra{} \partial D \text{ and }(1,0)\xlra{*} \partial D ] < \eps.
    	\end{align*}
\end{lem}

\begin{proof}
	Fix $p$ and $q$ and omit them from the notation. 
	Let $K,c$ be the constants given by Lemma~\ref{lem:Kcross} so that, for any $n>0$ and any boundary conditions $\xi$ on~$\La_{2n}$,
	\begin{align}\label{eq:KC1}
		\phi^{\xi}_{\La_{2n}}[\exists K \text{ disjoint clusters intersecting both $\La_n$ and $\partial \La_{2n}$}] \leq e^{-cn}.
	\end{align}
	By Lemma~\ref{lem:every-limit-is-gibbs2}, we may  pick $n$ sufficiently large that, 
	for any domain $D$ containing $\La_{n}$ and any boundary conditions $\xi$ on $D$ with at most $K$ wired components,
	\begin{align}\label{eq:KC2}
		\phi^{\xi}_{D}[0\lra \partial D \text{ and } (1,0)\xlra{*} \partial D] < \eps/2.
	\end{align}
	Moreover, we may assume that $e^{-cn}<\eps/2$. 
	
	Fix $N=2n$ and let $D$ be a domain containing $\La_{N}$ and $\xi$ be an arbitrary boundary condition on $\partial D$.
	Let $\calC$ be the union of the connected components in $D \setminus \La_n$ that intersect~$\partial D$.
	For any realisation $C$ of $\calC$, the restriction of $\phi_D^\xi[.|\, \calC = C]$ to $D\setminus C$  is 
	the measure $\phi_{D\setminus C}^\zeta$, where $\zeta$ is a boundary condition with as many wired components as there are clusters that intersect both $\partial D$ and $\La_n$. 
	By the choice of $K$,
	\begin{align*}
    	&\phi_{D}^\xi[0\xlra{} \partial D \text{ and }(1,0)\xlra{*} \partial D ] \\
		&\qquad\leq \sum_{C} \phi_{D\setminus C}^\zeta[0\xlra{} \partial (D\setminus C) \text{ and }(1,0)\xlra{*} \partial (D\setminus C) ] \, \phi_{D}^\xi[\calC = C]\\
		&\qquad\leq \eps/2 +
		\phi^{\xi}_{D}[\zeta \text{ contains more than $K$ wired components}] 
		\leq \eps.
    \end{align*}
	The first inequality is due to~\eqref{eq:KC2}, while the second to~\eqref{eq:KC1}.
\end{proof}

We are finally ready to proceed to the main goal of this section, namely the proof of  Corollary~\ref{cor:every-limit-is-gibbs}.

\begin{proof}[Corollary~\ref{cor:every-limit-is-gibbs}]
	As already mentioned, the proof is immediate when $\phi_{p,q}^{0}=\phi_{p,q}^{1}$.
	Fix $p\in (0,1)$ and $q\geq 1$ for which $\phi_{p,q}^0\neq \phi_{p,q}^1$ and let $\phi_{D_n,p,q}^{\xi_n}$ be a sequence of measures as in the statement, converging to some infinite-volume measure $\phi$. 
	
	The proof is essentially the same as that of Lemma~\ref{lem:every-limit-is-gibbs1}, 
	the only difference is that~\eqref{eq:bcbc} is unavailable when $\xi_n$ has more than a single wired component. 
	However, in light of Lemma~\ref{lem:every-limit-is-gibbs3}, for any $r$ such that $G \subset \La_r$ and $n$ sufficiently large for $D_n$ to contain $\La_{2r}$,
	\begin{align*}
		\phi_{D_n}^{\xi_n}\big[\zeta(\omega^{\xi_n}) \neq\zeta_{2r}(\omega) \big]
		&\leq 
		\sup_{\xi}\phi_{\La_{2r}}^{\xi}\big[G \xlra{} \partial \La_{2r} \text{ and } G \xlra{*} \partial \La_{2r}\big] \\
		&\leq 
		c_0(G)\,\sup_{\xi}\phi_{\La_{2r}}^{\xi}\big[0 \xlra{} \partial \La_{2r} \text{ and } (1,0)\xlra{*} \partial \La_{2r}\big]
		\xrightarrow[r \to \infty]{} 0,
	\end{align*}
	where the supremum is over all boundary conditions $\xi$ on $\partial \La_{2r}$,
	the first inequality is due to~\eqref{eq:SMP} and the second to the finite energy property; $c_0(G)$ is a constant depending on~$G$. 
	The rest of the proof is identical to that of Lemma~\ref{lem:every-limit-is-gibbs1}.
\end{proof}

\section{The Potts model}\label{sec:Potts_proof}

\subsection{Preparations}

\begin{lem}\label{lem:Phi2}
	Let $\mu$ be a tail trivial Gibbs measure for the Potts model on $\bbL$ with parameters $T>0$ and $q\geq 2$ 
	that is different from  $\mu^{\mathrm{free}}$, $\mu^{1}$, \dots, $\mu^{q}$.
	Then $\Phi$ is also tail trivial and $\omega$ contains at least two infinite clusters of different colours. 
\end{lem}

\begin{proof}
	Fix $T$ and $q$ and omit them from the notation. 
	The tail triviality of $\Phi$ follows readily from that of $\mu$ by Fubini's theorem because, 
	given the spin configuration, the edges are open or closed independently (with probability that depends on the spins).
	
	We turn to the existence of infinite clusters and proceed by contradiction. 
	Assume that, for $(\omega,\sigma)$ sampled from $\Phi$,
	all infinite clusters of $\omega$ have a.s. the same colour in $\sigma$.
	By the tail triviality of $\Phi$, this colour is constant -- assume it to be~$1$.
	Write $\phi$ for the $\omega$-marginal of $\Phi$. 
	
	Fix $n>0$ and consider $N > n$.
	Let us sample a pair of spin and edge configurations $(\sigma, \omega)$ from $\Phi$ in several steps.
	First we sample both configurations on $\La_{N}^{c}$.
	By the Gibbs property, the conditional distribution of $\omega$ inside $\La_N$ 
	is that of FK-percolation in $\La_N$ with wired boundary conditions, 
	conditioned on that no two points of $\partial \La_N$ that have different colour in $\sigma$ are connected (see Section~\ref{sec:ES}). 
	
	Sample now $\omega$ in $\La_N \setminus \La_n$ according to this distribution. 
	Let $\calE(n,N)$ be the event that no point of $\partial \La_N$ of colour different from $1$ is connected to $\La_n$. 
	Notice that $\calE(n,N)$ depends on $\omega$ in $\La_n^c$ and on $\sigma$ in $\La_N^c$. 
	Moreover, if $\calE(n,N)$ occurs, any completion of $\omega$ in $\La_n$ ensures that
	no two points of $\partial \La_N$ that have different colour in $\sigma$ are connected. 
	
	As in the previous section, write $\zeta(\omega)$ for the boundary condition on $\partial \La_n$ 
	induced by $\omega$, so that two vertices of $\partial \La_n$ are wired if and only if they are connected in $\omega \cap \La_n^c$ or both belong to infinite clusters.
	Then, by the Gibbs property for FK-percolation, the law of $\omega$ in $\La_n$ 
	conditionally on $\sigma$ on $\La_N^c$ and $\omega$ on $\La_n^c$ such that $\calE(n,N)$ occurs is~$\phi_{\La_{n}}^{\zeta(\omega)}$.

	Finally, by our assumption, $\Phi[\calE(n,N)] \xrightarrow[N\to \infty]{}1$.
	We conclude that there exists a sequence of random boundary conditions $\zeta_n$ for FK-percolation on $\La_n$ such that 
	$$ \phi = \lim_{n \to \infty }\phi_{\La_n}^{\zeta_n}.$$
	In light of Corollary~\ref{cor:every-limit-is-gibbs}, $\phi$ is a linear combination of $\phi^{0}$ and $\phi^{1}$.
	Moreover, by tail triviality, it is equal to one of these two measures. 
	It then follows that $\mu$ is either $\mu^{\rm free}$ or $\mu^{1}$, which contradicts our assumption.
\end{proof}

\subsection{Proof of  Theorem~\ref{thm:mainPotts}}

Let $\mu$ be an extremal Gibbs measure for the Potts model with parameters $T > 0$ and $q \ge 2$. 
As already discussed, the only case of interest is when $T = T_{c}(q)$ and $q$ is such that $\phi_{p_{c}(q),q}^0\neq \phi_{p_{c}(q),q}^1$. 
We remind the reader that this occurs if and only if $q >4$. 
For the whole proof $q$, $T$ and $p= p(T)$ are fixed such that $\phi_{p,q}^0\neq \phi_{p,q}^1$; they will be omitted from the notation. 

Let $\Phi$ be the Edwards-Sokal measure associated to $\mu$.
For $i \in \{1,\dots, q\}$, write $\omega^i$ for the edges of colour $i$ in $(\omega,\sigma)$, that is the edges $uv$ of $\omega$ such that $\sigma(u) = \sigma(v) = i$.

\begin{rem}
	It may be tempting to try to apply Theorem~\ref{thm:mainFK} to the $\omega$-marginal of~$\Phi$. 
	Notice however that this measure {\em a priori} might not be Gibbs for FK-percolation: 
	its restriction to finite domains is an FK-percolation measure {\em conditioned} that boundary vertices of different colours are disconnected. 
	
	We will rather consider the configurations $\omega^i$ and attempt to prove
	that at most one of these contains an infinite cluster.  Lemma~\ref{lem:Phi2} allows us then to conclude.
	
	The configurations $\omega^i$ do not have FK-percolation distributions, but will be shown to be dominated by certain 
	FK-percolation mesures, which due to Theorem~\ref{thm:mainFK} are of the form $\la_i \phi^1 + (1- \la_i)\phi^0$. 
	Then, using duality, we will prove that for any $i \neq j$, $\la_i + \la_j < 2$, 
	so that at most one of $\omega^i$ and $\omega^j$ contains a.s. an infinite cluster.  
\end{rem}

Fix $\delta > 0$ a small constant to be chosen below; it will be allowed to depend on $q$, but not on the values $n$ and $N$ appearing hereafter. 
Fix $n$ such that, for any $N \geq n$,
\begin{align}
	& \phi^1[\La_N \lra{} \partial \La_{2N}] > \tfrac12 \text{ and } \label{eq:choicen1}\\
	& \Phi[\exists K \text{ disjoint clusters intersecting $\partial \La_{2N}$ and $\La_N$}]  < \tfrac12, \label{eq:choicen2}
\end{align}
where $K$ is given by Lemma~\ref{lem:Kcross} \footnote{Lemma~\ref{lem:Kcross} does not apply exactly to $\Phi$, but the proof may easily be adapted.}.
It follows from  Corollary~\ref{cor:every-limit-is-gibbs} that we can fix $N = N(n) \geq 2n$ such that, for any boundary condition $\xi$ on $\partial \La_N$, 
\begin{align}\label{eq:choiceN}
	\inf_{\la \in [0,1]} \big\|(\phi_{\La_N}^\xi)_{|\La_{2n}} - ( \la \phi^1 + (1- \la)\phi^0)_{|\La_{2n}} \big\|_{\rm TV} \leq \delta,
\end{align}
where the above refers to the distance in total variation between the restrictions of $\phi_{\La_N}^\xi$ and $\la \phi^1 + (1- \la)\phi^0$ to $\La_{2n}$.

Define the domain $\calD$ explored from the outside in $\Phi$ as follows. 
Explore all interfaces between open and closed edges inside $\La_{2N} \setminus \La_N$ starting on $\partial \La_{2N}$
and ending either on $\partial \La_{2N}$ or on $\partial \La_{N}$; see Figure~\ref{fig:potts_flowers} for an illustration.  
Let $\calD$ be the connected component of $0$ in the set of unexplored edges.

We define $\tau : \partial \calD \to \{0,1,\dots, q\}$ as the boundary conditions induced by the configuration outside of $\calD$. 
More precisely, for each vertex $v \in \partial \calD$ which is on the free side of an exploration, set $\tau_{v} = 0$; 
if $v$ is on the wired side of an exploration, set $\tau_{v}$  for the colour of the primal open edges adjacent to $v$ in the exploration. 
Then, since $\calD$ is explored from the outside, for any realisation $D$ of $\calD$,
\begin{align*}
	\Phi[.\, |\, \calD=D, (\omega, \sigma)  \text{ on } D^{c}] _{|D}= \Phi_{D}^{\tau}.
\end{align*}
We emphasise that the event $\calD = D$ and the boundary conditions $\tau = \tau (\omega,\sigma)$  are determined by $(\omega,\sigma)$  on $D^{c}$.

Assume that $\calD= D$ and fix a realisation of $(\omega,\sigma) \text{ on } D^c$ and thus the boundary conditions $\tau$.
Notice that the boundary of $D$ is naturally split into a sequence of primal and dual arcs, with each of the primal arcs being monochromatic. 
For each colour $i$, write $D^{i}$ for the domain obtained from $D$ by deleting all edges which are adjacent to vertices $u \in \partial D$ with $\tau_{u}\notin \{0,i\}$.

We now define boundary conditions $\xi^{i}$ and $\zeta^{i}$ for FK-percolation on $D^{i}$. 
Notice that $\partial D^i$ shares vertices with $\partial D$, namely those on the primal arcs of $\partial D^i$ of colour $i$, as well those on the free arcs of $\partial D^i$.
In $\xi^{i}$, each primal arc of $\partial D$ of colour $i$ is wired, while the rest of the boundary is free.
In $\zeta^i$, all primal arcs of $\partial D$ of colour $i$ are wired together and the rest of the boundary is free. 
Thus $\xi^i$ and $\zeta^i$ differ only in the wiring between the different arcs of $\partial D$ of colour $i$. 
For simplicity, we will write $\phi_{D}^{\zeta^i}$ for the measure $\phi_{D^i}^{\zeta^i}$, 
where the edges in $D \setminus D^{i}$ are closed a.s.. The same applies to~$\xi^i$.

Below we state a series of results for the fixed realisation $D$ of $\calD$, with the fixed configurations $(\omega,\sigma)$ outside, 
and with the boundary conditions $\tau$, $\xi^i$ and $\zeta^i$ that these induce. 
	
\begin{figure}
    \begin{center}
        \includegraphics[width = 0.32\textwidth, page = 1]{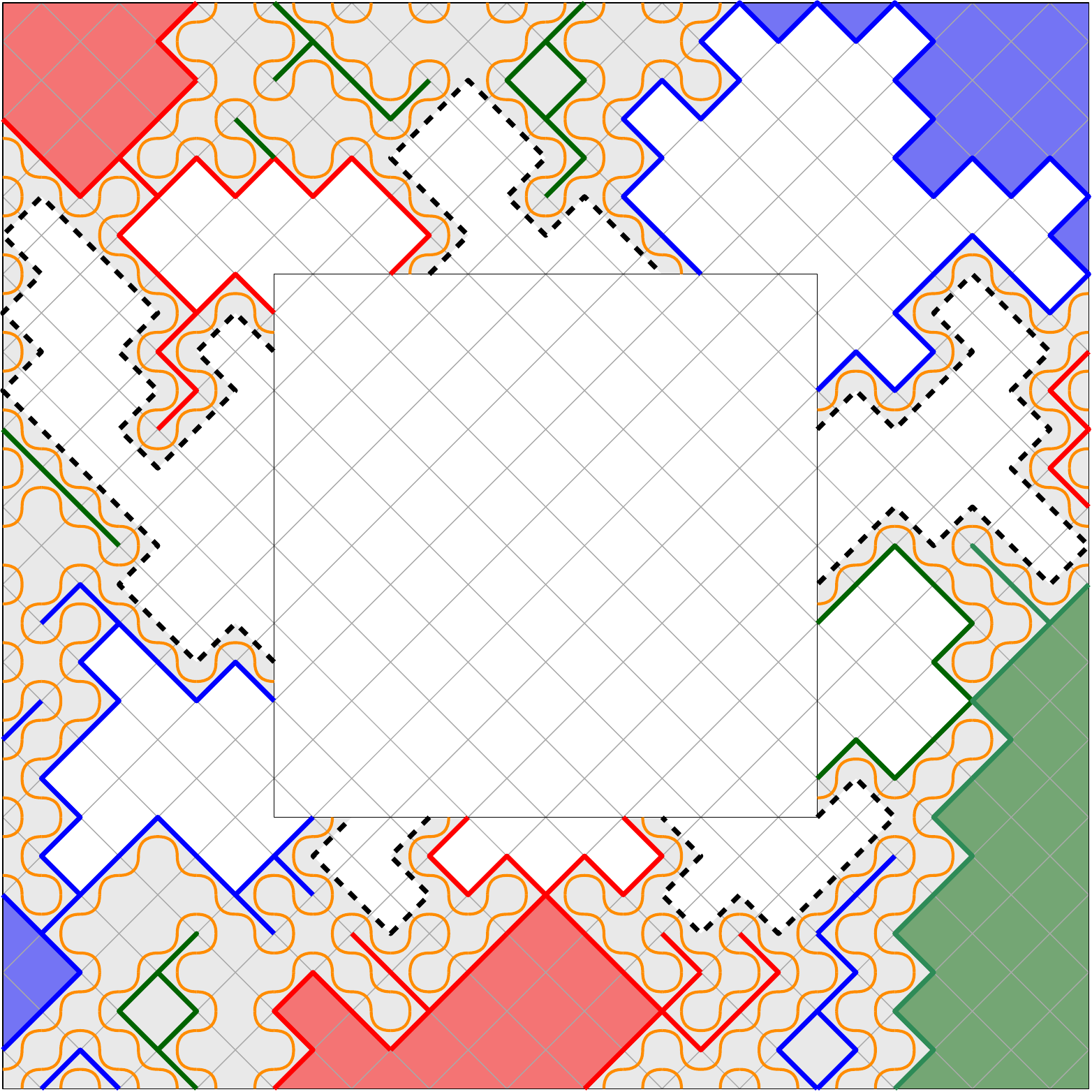}
        \includegraphics[width = 0.32\textwidth, page = 4]{potts_flower.pdf}
        \includegraphics[width = 0.32\textwidth, page = 3]{potts_flower.pdf}
        \caption{{\em Left:} the exploration of the interfaces originating on $\partial \La_{2N}$ reveals the edges in the grey region; 
            $\calD$ is formed of the edges entirely contained in  the white region.
            {\em Middle:}~The boundary conditions 
            $\zeta^i$ and $\xi^i$ induced on $\calD^{i}$ where $i$ corresponds to the blue colour. 
            The difference between them is that in $\zeta^{i}$ the two blue arcs are wired together, while in $\xi^i$ they are not. 
            {\em Right:} The boundary conditions $\xi^j$ induced on $\calD^{j}$ where $j$ corresponds to the red colour. 
            When considering the dual measure, the bold dotted contours correspond to wired boundary conditions. 
            Thus the dual measure dominates that in the middle picture (after shift by~$(1,0)$). }
        \label{fig:potts_flowers}
    \end{center}
\end{figure}

\begin{lem}\label{lem:p1}
	If $(\omega,\sigma)$ is sampled  according to  $\Phi_{D}^{\tau}$, 
	then the law of $\omega^{i}$ is dominated by~$\phi_{D}^{\zeta^{i}}$.
\end{lem}

\begin{proof}
	Let $(\omega,\sigma)$ be sampled according to $\Phi_{D}^{\tau}$ and
	write $\calC$ for the union of all clusters of $\omega$ of colours $j\neq i$ intersecting $\partial D$. 
	Then, for any realisation $C$ of $\calC$, the spatial Markov property for $\Phi$ ensures that the first marginal of 
	$\Phi_D^\tau(.|\calC = C)$ follows the FK-percolation measure $\phi_{D \setminus C}^\chi$, 
	where $\chi$ are the boundary conditions that are wired on the vertices of $\partial  (D \setminus C) \cap \partial D$ of colour $i$ in $\tau$, and free elsewhere. 
	This corresponds to the measure $\phi_{D}^{\zeta^i}$ conditioned on the event that all edges of, or adjacent to $C$ are closed. 
	As the latter is a decreasing event,~\eqref{eq:FKG} implies that the law of $\omega \setminus \calC$ is stochastically dominated by  $\phi_{D}^{\zeta^i}$. 
	Finally, by construction $\omega^i$ is contained in $\omega \setminus \calC$, which concludes the proof. 
\end{proof}

Write $\calK$ for the number of clusters of $\omega \cap D^c$ that intersect both $\partial \La_{2N}$ and $\partial \La_{N}$.

\begin{lem}\label{lem:p2}
	Let $K\in \bbN$ and assume that $(\omega,\sigma)$ outside of $D$ are such that $\calK = K$.
	Then the Radon-Nikodym derivative of $\phi_{D}^{\xi^{i}}$ with respect to $\phi_{D}^{\zeta^{i}}$ is bounded by $q^K$. 
	In particular, for any event $A$ depending only on the edges in $D$, 
	\begin{align*}
		\phi_{D}^{\zeta^{i}}[A] \leq q^K \phi_{D}^{\xi^{i}}[A].
	\end{align*}
\end{lem}

\begin{proof}
	As already stated, the only difference between $\xi^i$ and $\zeta^i$ is that the primal arcs of $\partial D$ of colour $i$ are wired together in $\zeta^{i}$ but not in $\xi^{i}$.
	There exist at most $K$ such arcs, as they are in bijection with the primal clusters in $\La_{2N} \setminus \La_{N}$ of colour $i$ which cross from $\partial \La_{2N}$ to $\partial\La_{N}$.
	Thus the weight of a configuration under $\phi^{\zeta^{i}}_{D}$ differs from that under $\phi^{\xi^{i}}_{D}$ by a multiplicative factor between $1$ and $q^{K}$.
	This also holds for the corresponding partition functions, and the statement follows.
\end{proof}

\begin{lem}\label{lem:p3}
	For $i \neq j$, if $\omega^{i}$ is sampled according to $\phi_{D}^{\xi^{i}}$,
	then the law of its dual configuration $(\omega^{i})^*$ translated by $(1,0)$
 	dominates $\phi_{D}^{\xi^{j}}$.
\end{lem}

In the above, we refer to the restrictions of the two measures to the intersection of $(D^{i})^* + (1,0)$ and $D^j$. 
We will actually only be interested in the effect of these two measures on the edges immediately adjacent to $0$,
 which are indeed contained in this intersection. 
 
\begin{proof}
	Fix $i \neq j$. 
	Write $E_i$ for the set of edges on the primal arcs of $\partial D$ of colour $i$ 
	(these edges are not part of $D$, only their endpoints are).
	Write $F_i$ for the set containing the edges whose duals form the dual arcs of $\partial D$,
	as well as all edges that are adjacent to the primal arcs of $\partial D$ of colours different from $i$. 
	Then $\omega^{i}$ has the law of $\phi^{0}$ restricted to $D$ and conditioned on that all edges of $E_i$ are open and all those of $F_i$ are closed and belong to the same dual cluster in $D^{c}$.
	The law of $\omega^{j}$ can be described in a similar way using the corresponding events $E_{j}$ and $F_{j}$:
	\[
		\omega_{j} \sim \phi^{0}[\cdot \,|\, \text{$\omega \equiv 1$ on $E_{j}$, $\omega \equiv 0$ on $F_{j}$, all edges of $F_{j}^*$ are connected in $\omega^* \cap D^c$}].
	\]
	
	Define the sets of edges $E_{i}'$ and $F_{i}'$ that are obtained from $E_{i}$ and $F_{i}$, respectively, 
	by considering the dual edges and translating these by $(1,0)$. 
	The above, along with the duality between $\phi^0$ and $\phi^1$ 
	imply that $(\omega^{i})^*$ translated by $(1,0)$ has the law of $\phi^{1}$ conditioned on that all edges of $E_{i}'$ are closed and those of $F_{i}'$ are open and belong to the same cluster in $D^{c}$:
	\[
		\omega_{i}^{*}+(1,0) \sim \phi^{1}[\cdot \,|\, \text{$\omega \equiv 0$ on $E_{i}'$, $\omega \equiv 1$ on $F_{i}'$, all edges of $F_{i}'$ are connected in $D^c$}].
	\]
	
	In order to prove that $\omega_{i}^{*}$ translated by $(1,0)$ dominates $\omega_{j}$1 it suffices to study the events on which the measures above are conditioned. 
	First observe that each edge of $E_i'$ may be written $e^* + (1,0)$ with $e \in E_i$. This means that all edges in $E_i'$ are incident to a vertex of $\partial D$ of colour $i$, whence $E_i' \subset F_j$.
	Applying the reverse transformation, we get similarly that $E_i \subset F_j'$.
	Finally, being in the same dual cluster is a decreasing condition and being in the same primal cluster is an increasing condition. The statement follows by the FKG inequality.
\end{proof}

For $i\in \{1,\dots, q\}$, write $\la_i = \la_i(D,\tau)$ for the value $\la \in [0,1]$ that minimises 
\begin{align}\label{eq:lambdasc}
	\big\|(\phi_{\calD}^{\xi^{i}})_{|\La_{2n}} - ( \la \phi^1 + (1- \la)\phi^0)_{|\La_{2n}} \big\|_{\rm TV}.
\end{align}

\begin{cor}\label{cor:lambdas}
	Let $e$ be the edge between $(0,0)$ and $(1,1)$. Then, for any $i,j  \in  \{1,\dots, q\}$ distinct, 
	\begin{align}\label{eq:lambdas}
		\la_i + \la_j \leq 1 + \frac{2\delta}{\phi^1[e \text{ open}] -\phi^0[e \text{ open}]}.
	\end{align}
\end{cor}

\begin{proof}
	Due to the choice of $N$ specified in \eqref{eq:choiceN}, the infimum of~\eqref{eq:lambdasc} is smaller than $\delta$.
	If $f$ denotes the edge between $(0,0)$ and $(-1,1)$,
	Lemma~\ref{lem:p3} implies that 
	\begin{align*}
    		1 
    		& =  \phi_{\calD}^{\xi^{1}}[e \text{ open}] + \phi_{\calD}^{\xi^{1}}[e \text{ closed}] \\
    		&\geq \phi_{\calD}^{\xi^{1}}[e \text{ open}] + \phi_{\calD}^{\xi^{2}}[f \text{ open}] 
    		\geq (\la_1 + \la_2) \phi^1[e \text{ open}] + (2- \la_1 - \la_2)\phi^0[e \text{ open}] - 2\delta,
	\end{align*}
	where we used that $\phi^0[e \text{ open}] = \phi^0[f \text{ open}]$.
	Moreover, since $\phi^0$ is dual to $\phi^1$,  $\phi^1[e \text{ open}] = 1 -\phi^0[e \text{ open}] >1/2$, and basic algebra yields~\eqref{eq:lambdas}.
\end{proof}

Let us now conclude the proof of the theorem.
	
\begin{proof}[Theorem~\ref{thm:mainPotts}]
	By the choice~\eqref{eq:choicen2} of $n$, for any $i\in \{1,\dots,q\}$,
	\begin{align*}
		\Phi[\La_n \nxlra{~\omega^i\,\,\,~} \La_{2n}^c]
		&\geq \Phi\big[\phi_{\calD}^{\zeta^{i}}[\La_n \nxlra{~\omega^i\,\,\,~} \La_{2n}^c] \big]\\
		&\geq \tfrac12 \Phi\big[\phi_{\calD}^{\zeta^{i}}[\La_n \nxlra{~\omega^i\,\,\,~} \La_{2n}^c] \,\big|\, \calK \leq K \big]\\
		&\geq \tfrac12 q^{-K} 
		\Phi\big[ \phi_{\calD}^{\xi^{i}}[\La_n \nxlra{~\omega^i\,\,\,~} \La_{2n}^c] \,\big|\, \calK \leq K \big] \\
		&\geq \tfrac12 q^{-K} 
		\Phi\big[ \la_i \phi^1[\La_n \nxlra{~\omega^i\,\,\,~} \La_{2n}^c] + (1- \la_i)\phi^0[\La_n \nxlra{} \La_{2n}^c] - \delta \,\big|\, \calK \leq K \big] \\
		&\geq \tfrac12 q^{-K}\Big( \Phi[1 - \la_i\,|\,\calK \leq K ] \phi^0[\La_n \nxlra{} \La_{2n}^c] - \delta\Big),
	\end{align*}
	where $\Phi$ governs the choice of $\calD$ and $\tau$, 
	and hence of $\zeta^{i}$, $\xi^i$ and ultimately $\la_i=\la_i(\calD,\tau)$.
	In the above, the first line uses Lemma \ref{lem:p1}, 
	the second uses~\eqref{eq:choicen2}, the third uses Lemma \ref{lem:p2}, and the fourth uses the definition of $\lambda_{i}$ and the choice \eqref{eq:choiceN} of $N$.
	
	We proceed by contradiction and assume that $\mu$ is different from $\mu^{\mathrm{free}}$ and $\mu^i$ with $i \in \{1,\dots, q\}$.
	Then, due to Lemma~\ref{lem:Phi2}, $\Phi$ contains at least two infinite clusters of different colours. 
	Permute the colours so that there exist infinite clusters of colours $1$ and $2$. 

	Then, due to Corollary~\ref{cor:lambdas}, 
	\begin{align*}
		\Phi[\la_1\,|\,\calK \leq K] +\Phi[\la_2\,|\,\calK \leq K] \leq 1 + \frac{2\delta}{\phi^1[e \text{ open}] -\phi^0[e \text{ open}]} \leq 3/2,
	\end{align*}
	provided that $\delta$ is small enough. 
	Thus, at least one of $\Phi[\la_1|\calK \leq K]$ and $\Phi[\la_2 |\calK \leq K]$ is smaller than $3/4$. 
	Assume with no loss of generality that it is the former.

	Recall that, since there is no percolation under $\phi^0$, $\phi^0[\La_n \nxlra{} \La_{2n}^c] \to 1$ as $n\to1$. 
	Thus, assuming $\delta < 1/4$, we find that $\limsup_{n \to \infty}\Phi[\La_n \xlra{\omega^1} \La_{2n}^c] < 1$.
	This contradicts our assumption that there exists percolation in $\omega^1$ $\Phi$-a.s.. 
\end{proof}

\section{Extensions}\label{sec:extensions}

In this section we describe how the proof of Theorem~\ref{thm:mainFK} may be extended to two other settings, namely FK-percolation on the triangular and hexagonal lattices and the loop $O(n)$ model. Since the argument requires only minor modifications, we only provide sketches. 

\subsection{FK-percolation and Potts models on the triangular and hexagonal lattices}\label{sec:Hex}

Consider the triangular and hexagonal lattices of Figure~\ref{fig:tri} defined in the plane $\bbC$ as follows. 
The triangular lattice, denoted ${\rm Tri}$, has vertex set $\{u + v e^{i\pi/3}: u,v \in \bbZ\}$ and edges between vertices at euclidian distance $1$. 
The hexagonal lattice ${\rm Hex}$ is the dual of ${\rm Tri}$: it has a vertex in the centre of each face of ${\rm Tri}$ and edges between vertices corresponding to neighbouring faces. 

\begin{figure}
    \begin{center}
        \raisebox{-0.5\height}{\includegraphics[width = 0.4\textwidth]{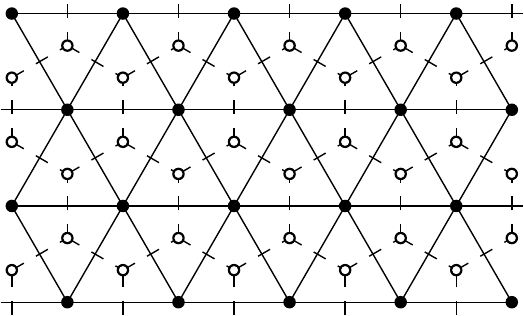}}\qquad
        \raisebox{-0.5\height}{\includegraphics[width = 0.5\textwidth]{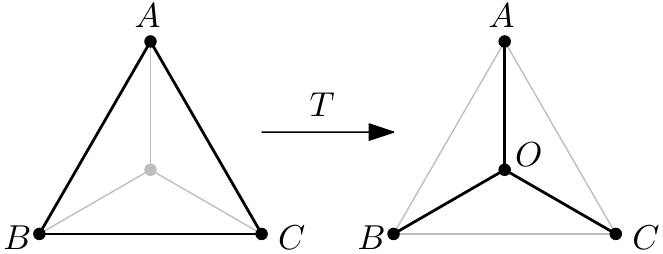}}
	\caption{
        	{\em Left:} The triangular lattice ${\rm Tri}$ and its dual hexagonal lattice ${\rm Hex}$. 
       		{\em Right:} The star-triangle transformation $T$ transforms the triangle $\bf T$ into the star $\bf S$ by adding a central vertex $O$
        		and replacing the edges between $A$, $B$ and $C$ by those between $O$ and  $A$, $B$ and $C$. 
	}
        \label{fig:stt}\label{fig:tri}
    \end{center}
\end{figure}

The FK-percolation models on ${\rm Tri}$ and ${\rm Hex}$ are defined similarly to that on $\bbL$. 
In particular, one may define infinite-volume measures $\phi^{i}_{{\rm Tri}, p,q}$ and $\phi^{i}_{{\rm Hex}, p,q}$ via thermodynamical limits
for $i \in \{0,1\}$, $q \geq 1$ and $p \in (0,1)$. 
Moreover, the duality relation described in Section~\ref{sec:fk-background} adapts to this setting and relates the models on ${\rm Tri}$ and ${\rm Hex}$, 
with the edge-weights related by~\eqref{eq:pp*}.
In spite of the primal and dual models acting on different lattices, a notion of self-duality may be defined in this context using the star-triangle transformation. 

\begin{lem}\label{lem:stt}
	Consider the triangle and star graphs ${\bf T}$ and ${\bf S}$ of Figure~\ref{fig:stt} with $\partial {\bf T} = \partial {\bf S} = \{A,B,C\}$. 
	For any $q \geq 1$, any boundary conditions $\xi$, for $p$ such that 
	\begin{align}\label{eq:p_c_tri}
		\Big(\tfrac{p}{1-p}\Big)^3 + 3\Big(\tfrac{p}{1-p}\Big)^2 = q, 
	\end{align}
	there exists a coupling $\bbP$ of $\omega' \sim \phi_{{\bf T},p,q}^\xi$ and 
	 $\omega \sim \phi_{{\bf S},p^*,q}^\xi$ such that 
	 \begin{align*}
	 	U \xlra{\omega} V  \Leftrightarrow 	 	U \xlra{\omega'} V  \quad \forall U,V \in \{A,B,C\},\, \bbP\text{-a.s.}.
	 \end{align*}
\end{lem}

The coupling above may be described as a random map $T$ associating to a configuration $\omega$ the configuration $\omega'$. 
When $\omega$ contains at least one open edge, there exists a unique configuration $\omega'$ that yields the same connections between $A$, $B$ and $C$ as $\omega$, and we set $T(\omega) = \omega'$. 
When $\omega$ contains no open edges, $T(\omega)$ is chosen at random between the configurations on ${\bf S}$ with zero or one open edges: 
the configuration with zero open edges is chosen with probability $(\tfrac{p}{1-p})^3/q$, while each configuration with exactly one open edge is chosen with probability  $(\tfrac{p}{1-p})^2/q$.
Finally, the randomness in $T$ is independent of that in $\omega$. 

The proof of Lemma~\ref{lem:stt}, as well as of the explicit coupling described above appeared in several papers, for instance in~\cite[Prop. 2.2]{GriMan13}.

Using the star-triangle transformation, it was proved in~\cite{BefDum12} that~\eqref{eq:p_c_tri} determines the critical point $p_c({\rm Tri}, q)$ of FK-percolation on the triangular lattice, 
and that the critical point for the hexagonal lattice is given by its dual $p_c({\rm Hex}, q)=p_c({\rm Tri}, q)^*$.
Moreover, the dichotomy of~\cite{DumSidTas17} adapts to this setting. Finally, when $p\in (0,1)$ and $q\geq 1$, it was shown in~\cite{DumLiMan18} that
\begin{align*}
   \phi_{{\rm Tri},p,q}^0\neq \phi_{{\rm Tri}, p,q}^1 &\Longleftrightarrow q > 4 \text{ and } p = p_c({\rm Tri}, q)\qquad  \text{ and }\\
   \phi_{{\rm Hex},p,q}^0\neq \phi_{{\rm Hex},p,q}^1 &\Longleftrightarrow q > 4 \text{ and } p = p_c({\rm Hex}, q).
\end{align*}

The proof of~\cite{BefDum12} is based on the observations that, for $p =  p_c({\rm Tri}, q)$, 
the star-triangle transformation $T$ may be applied to each upward pointing triangle in ${\rm Tri}$
and the resulting graph is ${\rm Hex}$ shifted up by $1/\sqrt3$.
Moreover, if the configuration on ${\rm Tri}$ is chosen according to $\phi^{i}_{{\rm Tri}, p,q}$ for some $i \in \{0,1\}$, 
then the resulting one is the shift of $\phi^{i}_{{\rm Hex}, p^*,q}$ by $(0,-1/\sqrt3)$

As stated in the introduction, Theorem~\ref{thm:mainFK} extends to the triangular and hexagonal lattices. 

\begin{thm}\label{thm:mainFK_tri}
	Let $\phi$ be a Gibbs measure for FK-percolation on ${\rm Tri}$ or ${\rm Hex}$ with edge-intensity $p \in (0,1)$ and cluster-weight $q \geq 1$. 
	Then there exists $\la \in [0,1]$ such that 
	\begin{align*}
		\phi = \la \phi_{{\rm Tri},p,q}^0 + (1-\lambda)  \phi_{{\rm Tri},p,q}^1 \quad \text{ or } \quad \phi = \la \phi_{{\rm Hex},p,q}^0 + (1-\lambda)  \phi_{{\rm Hex},p,q}^1,
	\end{align*}
	depending on the setting. 
\end{thm}

The proof is exactly similar to that of Theorem~\ref{thm:mainFK}, with only one notable difference. We may focus on ${\rm Tri}$, as the result for ${\rm Hex}$ is obtained by duality. Only the setting $p = p_c({\rm Tri})$ requires attention; for all other values of $p$ the Gibbs measure is unique. 

In  Proposition~\ref{prop:no_hp_perco} and Lemma~\ref{lem:perc-hp-some-good}, we used an automorphism $\rho$ of the set of configurations
which was mapping the primal model to an (approximate) vertical symmetry of the dual model. 
In the present case, $\rho$ should be replaced with the composition of the application of $T$ to each upward pointing triangle, 
the symmetry with respect to the horizontal axis $\bbR \times \{0\}$, and the shift by $(0,1/\sqrt 3)$.
Using this transformations, the proofs of both Proposition~\ref{prop:no_hp_perco} and Lemma~\ref{lem:perc-hp-some-good} adapt readily. The rest of the proof is identical. 

In the same way as in Corollary~\ref{cor:every-limit-is-gibbs}, it may be deduced from Theorem~\ref{thm:mainFK_tri}
that any thermodynamical limit of FK-percolation measures on the triangular and hexagonal lattices is Gibbs. 

Finally, the proof of Theorem~\ref{thm:mainPotts} may also be adapted to the triangular and hexagonal lattices, with only minor modifications. In addition to using Theorem~\ref{thm:mainFK_tri} instead of Theorem~\ref{thm:mainFK}, one only needs to adapt the proof of Lemma~\ref{lem:p3}, 
where auto-duality plays a particular role. This may easily be done using duality combined with the star-triangle transformation, as descried above.

\subsection{Proof for loop $O(n)$ model}\label{sec:loop_proof}

Fix for the whole section $n\geq 1$ and $0 < x \leq {1}/{\sqrt n}$. 

We start by considering the structure of Gibbs measures for the spin system (see the discussion of Section~\ref{sec:loop-background} for details). 
As for FK-percolation, using monotonicity, one may define the Gibbs measures ${\sf Spin}_{n,x}^{+}$ and ${\sf Spin}_{n,x}^{-}$ for the spin system 
as limits of finite-volume measures. By construction, these are translation invariant and tail trivial.
Moreover, due to~\eqref{eq:On_mono}, any Gibbs measure for the spin system is contained between these two for the stochastic ordering. 

The relevant result for the spins system is the following

\begin{prop}\label{prop:GibbsSpin}
	Any Gibbs measure for the spin system with parameters $(n,x)$ 
	is a linear combination of ${\sf Spin}_{n,x}^{+}$ and ${\sf Spin}_{n,x}^{-}$.
\end{prop}

\begin{proof}
We only sketch this proof as it is very similar to that of Theorem~\ref{thm:mainFK}. 
If the parameters are such that the model is in the macroscopic loops regime, then ${\sf Spin}_{n,x}^{+} = {\sf Spin}_{n,x}^{-}$, 
and this is the only Gibbs measure for the spin system. 

If the model is in the exponential decay regime, then ${\sf Spin}_{n,x}^{+} \neq {\sf Spin}_{n,x}^{-}$, 
but the proofs of Sections~\ref{sec:FK_preparation} and~\ref{sec:FK_duplication} apply with only minor changes. 
Indeed, the only significant alteration is in  the transformation $\rho$ used in the proofs of Proposition~\ref{prop:no_hp_perco} and Lemma~\ref{lem:perc-hp-some-good}:
it should be replaced here with a vertical symmetry composed with a sign flip. 
\end{proof}

\begin{rem}
	The proof for the spin system may appear closer to those for the Ising model from~\cite{Aiz80},~\cite{Hig81} or~\cite{GeoHig00}.
	Note however that the boundary conditions in the DLR equation for the spin system are not determined entirely by the spins on the boundary of the domain, 
	and therefore the more complex argument of the present paper is necessary.
\end{rem}

Next we turn to the Gibbs measures for the loop model. The key is the following lemma. 

\begin{lem}\label{lem:LtoS}
	Let ${\sf Loop}$ be a Gibbs measure for the loop model with parameters $(n,x)$. 
	Consider the measure ${\sf Spin}$ on $ \{-1,1\}^{F({\rm Hex})}$ obtained by sampling a configuration $\omega$ from ${\sf Loop}$, 
	then sampling uniformly one of the two spin configurations coherent with $\omega$. 
	Then ${\sf Spin}$ is a Gibbs measure for the spin system with parameters $(n,x)$.
\end{lem}

\begin{proof}
	Consider a domain $D$.
	Recall that a boundary condition $\tau$ for the spin system on $D$ 
	contains an assignment of spins for the faces neighbouring $D$ and connectivity information between these faces in the spin percolation. 
	A boundary condition for the loop model on $D$ is a pairing of certain vertices of $\partial D$, acting as the endpoints of the various segments of loops in $D^c$. 
	
	It is immediate from the correspondence between loops and spins that
	there is a two-to-one function mapping spin boundary conditions $\tau$ on loop boundary conditions $\xi(\tau)$.
	Precisely, for any coherent loop and spin configuration $\omega,\sigma$ on the full lattice, we have
	\begin{align*}
		\omega \text{ induces $\xi(\tau)$ on $D$}\, \Leftrightarrow \, \sigma \text{ induces $\tau$ or $-\tau$ on $D$}.
	\end{align*}

	Fix a spin boundary condition $\tau$ and write $\xi =\xi(\tau)$ for the corresponding loop boundary condition. 
	Let $\omega$ be a random configuration chosen according to ${\sf Loop}$ and $\sigma$ be a spin configuration chosen uniformly between the two configurations coherent with $\omega$; $\sigma$ has law ${\sf Spin}$. 
	Conditioning on that $\sigma$ induces $\tau$ is the same as conditioning that $\omega$ induces $\xi$ and that the choice of the spins is such that $\sigma$ corresponds to $\tau$ on the faces neighbouring $D$. 
	Since the latter choice is independent of $\omega$, 
	we conclude that the law of $\omega$ in $D$ conditioned on $\{\sigma \text{ induces } \tau\}$ 
	is the same as that conditioned on $\{\omega \text{ induces } \xi\}$, 
	and is therefore ${\sf Loop}_{D,n,x}^{\xi}$ due to the Gibbs property for  ${\sf Loop}$. 
	
	Now, due to the properties of the correspondence between loop and spin configurations \cite[Prop. 3]{DumGlaPel21}, 
	we conclude that 
	$${\sf Spin}[.\,|\, \sigma \text{ induces } \tau] = {\sf Spin}_{D,n,x}^{\tau}.$$

	The same holds when conditioning additionally on any event depending only on the spins on $D^c$, which proves the DLR property for ${\sf Spin}$.   
\end{proof}

\begin{proof}[Theorem~\ref{thm:loop}]		
	We start with the existence. 
	A Gibbs measure for the loop model may be obtained by considering the paths separating $+$ from $-$ spins in a configuration chosen according to ${\sf Spin}_{n,x}^{+}$. 
	Since ${\sf Spin}_{n,x}^{+}$ is translation invariant and tail trivial, the same holds for the associated loop measure. 
	
	We turn to the uniqueness. Let ${\sf Loop}$ be a Gibbs measure for the loop model, and construct ${\sf Spin}$ as in Lemma~\ref{lem:LtoS}.
	Then, due to Proposition~\ref{prop:GibbsSpin}, ${\sf Spin} = \la {\sf Spin}_{n,x}^{+} + (1-\lambda) {\sf Spin}_{n,x}^{-}$ for some $\la \in [0,1]$.
	
	Observe now that ${\sf Spin}_{n,x}^{-}$ may be obtained from ${\sf Spin}_{n,x}^{+}$ by a global sign flip. 
	As such, the laws of the interfaces between $+$ and $-$ clusters in spin configurations sampled from 
	${\sf Spin}_{n,x}^{+}$ and ${\sf Spin}_{n,x}^{-}$, respectively, are the same. 
	In conclusion, ${\sf Loop}$ does not depend on $\la$, and is therefore unique. 
\end{proof}

\bibliographystyle{abbrv}
\bibliography{biblicomplete}

\end{document}